\documentclass[11pt, final]{article}
\usepackage{a4}

\usepackage{amsmath}
\usepackage{amssymb}
\usepackage{amsxtra}
\usepackage{amsthm}
\usepackage{bbm}			
\usepackage{cite}

\usepackage[left=3.5cm,right=3.5cm,top=3cm,bottom=3cm]{geometry}

\usepackage[usenames,dvipsnames]{xcolor}
\usepackage[plainpages=false,pdfpagelabels,colorlinks=true,linkcolor=blue,citecolor=Magenta]{hyperref} 

\usepackage{showkeys}
\usepackage[latin1]{inputenc}       
\usepackage[T1]{fontenc}        
\usepackage{lmodern}            
\usepackage{graphicx}
\usepackage{subfig}		
\usepackage{booktabs}		
\usepackage{multirow}
\usepackage{enumitem}		
\usepackage{listings} 	
\usepackage{mathtools}  
\usepackage{floatrow}		
\usepackage{caption}

\usepackage{colortbl}
\definecolor{dunkelgrau}{rgb}{0.8,0.8,0.8}
\definecolor{hellgrau}{rgb}{0.9,0.9,0.9}

\newcommand{\R}{\ensuremath{\mathbb{R}}}

\newcommand{\oR}{\ensuremath{\overline{\mathbb{R}}}}

\renewcommand{\>}{\right\rangle}
\newcommand{\<}{\left\langle}
\newcommand{\id}{\ensuremath{\text{Id}}}

\newcommand{\bx}{\ensuremath{\overline{x}}}
\newcommand{\by}{\ensuremath{\overline{y}}}

\newcommand{\bv}{\ensuremath{\overline{v}}}

\newcommand{\bp}{\ensuremath{\overline{p}}}

\newcommand{\bu}{\ensuremath{\overline{u}}}

\renewcommand{\bv}{\ensuremath{\overline{v}}}

\newcommand{\f}{\ensuremath{\boldsymbol}}
\newcommand{\fx}{\ensuremath{\boldsymbol{x}}}
\newcommand{\fv}{\ensuremath{\boldsymbol{v}}}
\newcommand{\fbv}{\ensuremath{\boldsymbol{\overline{v}}}}

\newcommand{\fbx}{\ensuremath{\boldsymbol{\overline{x}}}}
\newcommand{\fby}{\ensuremath{\boldsymbol{\overline{y}}}}
\newcommand{\fp}{\ensuremath{\boldsymbol{p}}}
\newcommand{\fq}{\ensuremath{\boldsymbol{q}}}

\newcommand{\fy}{\ensuremath{\boldsymbol{y}}}
\newcommand{\fz}{\ensuremath{\boldsymbol{z}}}
\newcommand{\fw}{\ensuremath{\boldsymbol{w}}}

\newcommand{\fG}{\ensuremath{\boldsymbol{\mathcal{G}}}}

\newcommand{\fK}{\ensuremath{\boldsymbol{\mathcal{K}}}}

\newcommand{\g}{\ensuremath{\mathcal{G}}}
\newcommand{\h}{\ensuremath{\mathcal{H}}}

\newcommand{\Y}{\ensuremath{\mathcal{Y}}}
\newcommand{\proj}{\ensuremath{\mathcal{P}}}

\renewcommand{\Box}{\ensuremath{\mbox{\small$\,\square\,$}}}

\theoremstyle{plain}
\newtheorem{theorem}{Theorem}[section]

\theoremstyle{definition}
\newtheorem{remark}{Remark}[section]
\newtheorem{example}{Example}[section]

\newtheorem{problem}{Problem}[section]
\newtheorem{algorithm}{Algorithm}[section]

\DeclareMathOperator*\dom{dom}%
\DeclareMathOperator*\gra{gra}%
\DeclareMathOperator*\argmin{arg\,min}%
\DeclareMathOperator*\cl{cl}%
\DeclareMathOperator*\ran{ran}%
\DeclareMathOperator*\zer{zer}%
\DeclareMathOperator*\Fix{fix}%
\DeclareMathOperator*\Prox{Prox}%

\numberwithin{equation}{section}  

\title{A Douglas-Rachford type primal-dual method for solving inclusions with mixtures of composite and parallel-sum type monotone operators}

\author{Radu Ioan Bo\c t
\thanks {Department of Mathematics, Chemnitz University of Technology, D-09107 Chemnitz, Germany, e-mail: radu.bot@mathematik.tu-chemnitz.de. Research partially supported by DFG (German Research Foundation), project BO 2516/4-1.}
\and Christopher Hendrich
\thanks{Department of Mathematics, Chemnitz University of Technology, D-09107 Chemnitz, Germany, e-mail: christopher.hendrich@mathematik.tu-chemnitz.de. Research supported by a Graduate Fellowship of the Free State Saxony, Germany.}
}
\date{\today}

\begin{document}
\maketitle

{\bf Abstract.} In this paper we propose two different primal-dual splitting algorithms for solving inclusions involving mixtures of composite and parallel-sum type monotone operators which rely on an inexact Douglas-Rachford splitting method, however applied in different underlying Hilbert spaces. Most importantly, the algorithms allow to process the bounded linear operators and the set-valued operators occurring in the formulation of the monotone inclusion problem separately at each iteration, the latter being individually accessed via their resolvents. The performances of the primal-dual algorithms are emphasized via some numerical experiments on location and image deblurring problems.



\section{Introduction and preliminaries}\label{sectionIntro}

In applied mathematics, a wide range of convex optimization problems such as single- or multifacility location problems, support vector machine problems for classification and regression, portfolio optimization problems as well as signal and image processing problems, all of them likely possessing nondifferentiable convex objectives, can be reduced to the solving of inclusions involving mixtures of monotone set-valued operators.

In this article we propose two different primal-dual iterative error-tolerant methods for solving inclusions with mixtures of composite and parallel-sum type monotone operators. Both algorithms rely on the inexact Douglas-Rachford algorithm (cf. \cite{Com04,Com09}), but still differ clearly from each other. An important feature of the two approaches and, simultaneously, an advantage over many existing methods is their capability of processing the set-valued operators separately via their resolvents, while the bounded linear operators are accessed via explicit forward steps on their own or on their adjoints. The resolvents of the maximally monotone operators are not always available in closed form expressions, fact which motivates the inexact versions of the algorithms, where implementation errors in the shape of summable sequences are allowed.

The methods in this article are also perfectly parallelizable since the majority of their steps can be executed independently. Furthermore, when applied to subdifferential operators of proper, convex and lower semicontinuous functions, the solving of the monotone inclusion problems is, under appropriate qualification conditions (cf. \cite{Bot10,BGW09}), equivalent with finding optimal solutions to primal-dual pairs of convex optimization problems. The considered formulation also captures various different types of primal convex optimization problems and corresponding conjugate duals appearing in wide ranges of applications. The resolvents of subdifferentials of proper, convex and lower semicontinuous functions are the proximal point mappings of these and are known to take in a lot of situations shape of closed form expressions.

Recent research (see \cite{BotCseHein12, BotHend12, BriCom11,ComPes12,Vu11}) has shown that structured problems dealing with monotone inclusions can be efficiently solved via primal-dual splitting approaches. In \cite{ComPes12}, the problem involving sums of set-valued, composed, Lipschitzian and parallel-sum type monotone operators was decomposed and solved via an inexact Tseng algorithm having foward-backward-forward characteristics in a product Hilbert space. On the other hand, in \cite{Vu11}, instead of Lipschitzian operators, the author has assumed cocoercive operators and solved the resulting problem with an inexact forward-backward algorithm. Thus, our methods can be seen as a continuation of these ideas, this time by making use of the inexact Douglas-Rachford method. Another primal-dual method relying on the same fundamental splitting algorithm is considered in \cite{Con12} in the context of solving minimization problems having as objective the sum of two proper, convex and lower semicontinuous functions, one of them being composed with a bounded linear operator.

Due to the nature of Douglas-Rachford splitting, we will neither assume Lipschitz continuity nor cocoercivity for any of the operators present in the formulation of the monotone inclusion problem. The resulting drawback of not having operators which can be processed explicitly via forward steps is compensated by the advantage of allowing general maximal monotone operators in the parallel-sums, fact which relaxes the working hypotheses in \cite{ComPes12,Vu11}.

The article is organized as follows. In the remaining of this section we introduce the framework we work within and some necessary notations. The splitting algorithms and corresponding weak and strong convergence statements are subject of Section \ref{sectionInc} while Section \ref{sectionMin} is concerned with the application of the two methods to convex minimization problems. Finally, in Section \ref{sectionEx} we make some numerical experiments and evaluate the obtained results.

We are considering the real Hilbert spaces $\h$ and $\g_i$ endowed with the \textit{inner product} $\left\langle \cdot ,\cdot \right\rangle_{\h}$ and $\left\langle \cdot ,\cdot \right\rangle_{\g_i}$ and associated \textit{norm} $\left\| \cdot \right\|_{\h} = \sqrt{\left\langle \cdot, \cdot \right\rangle_{\h}}$ and $\left\| \cdot \right\|_{\g_i} = \sqrt{\left\langle \cdot, \cdot \right\rangle_{\g_i}}, i=1,\ldots,m$, respectively. The symbols $\rightharpoonup$ and $\rightarrow$ denote weak and strong convergence, respectively, $\R_{++}$ denotes the set of strictly positive real numbers and $\R_+ = \R_{++} \cup \{0\}$. By $B(0,r)$ we denote the closed ball with center $0$ and radius $r \in \R_{++}$. For a function $f: \h \rightarrow \oR = \R \cup \{\pm \infty\}$ we denote by $\dom f := \left\{ x \in \h : f(x) < +\infty \right\}$ its \textit{effective domain} and call $f$ \textit{proper} if $\dom f \neq \varnothing$ and $f(x)>-\infty$ for all $x \in \h$. Let be
$$\Gamma(\h) := \{f: \h \rightarrow \overline \R: f \ \mbox{is proper, convex and lower semicontinuous}\}.$$
The \textit{conjugate function} of $f$ is $f^*:\h \rightarrow \oR$, $f^*(p)=\sup{\left\{ \left\langle p,x \right\rangle -f(x) : x\in\h \right\}}$ for all $p \in \h$ and, if $f \in \Gamma(\h)$, then $f^* \in \Gamma(\h)$, as well. The \textit{(convex) subdifferential} of $f: \h \rightarrow \oR$ at $x \in \h$ is the set $\partial f(x) = \{p \in \h : f(y) - f(x) \geq \left\langle p,y-x \right\rangle \ \forall y \in \h\}$, if $f(x) \in \R$, and is taken to be the empty set, otherwise. For a linear continuous operator $L_i: \h \rightarrow \g_i$, the operator $L_i^*: \g_i \rightarrow \h$, defined via $\< L_ix,y  \> = \< x,L_i^*y  \>$ for all $x \in \h$ and all $y \in \g_i$, denotes its \textit{adjoint}, for $i\in \{1,\ldots,m\}$.

Having two functions $f,\,g : \h \rightarrow \oR$, their \textit{infimal convolution} is defined by $f \Box g : \h \rightarrow \oR$, $(f \Box g) (x) = \inf_{y \in \h}\left\{ f(y) + g(x-y) \right\}$ for all $x \in \h$, being a convex function when $f$ and $g$ are convex.

Let $M:\h \rightarrow 2^{\h}$ be a set-valued operator. We denote by $\zer M = \{ x \in \h : 0 \in M x \}$ its set of \textit{zeros}, by $\Fix M = \{x \in \h: x \in Mx\}$ its set of \textit{fixed points}, by $\gra M = \{ (x,u) \in \h \times \h : u \in Mx\}$ its \textit{graph} and by $\ran M =\{u \in \h : \exists x \in \h,\ u\in Mx\}$ its \textit{range}. The \textit{inverse} of $M$ is $M^{-1}:\h \rightarrow 2^{\h}$, $u \mapsto \{ x\in\h : u \in Mx \}$. We say that the operator $M$ is \textit{monotone} if $\< x-y,u-v \> \geq 0$ for all $(x,u),\,(y,v) \in \gra M$ and it is said to be \textit{maximally monotone} if there exists no monotone operator $M':\h \rightarrow 2^{\h}$ such that $\gra M'$ properly contains $\gra M$. The operator $M$ is said to be \textit{uniformly monotone} with modulus $\phi_M : \R_{+} \rightarrow [0,+\infty]$ if $\phi_M$ is increasing, vanishes only at $0$, and $\< x-y,u-v \> \geq \phi_M \left( \| x-y \|\right)$ for all $(x,u),\,(y,v) \in \gra M$.

The \textit{resolvent} and the \textit{reflected resolvent} of an operator  $M:\h \rightarrow 2^{\h}$ are
$$ J_M = \left( \id + M \right)^{-1} \text{ and }R_M = 2J_M - \id,$$
respectively, the operator $\id$ denoting the identity on the underlying Hilbert space. When $M$ is maximally monotone, its resolvent (and, consequently, its reflected resolvent) is a single-valued operator and, by \cite[Proposition 23.18]{BauCom11}, we have for $\gamma \in \R_{++}$
\begin{align}
	\label{res-identity}
	\id = J_{\gamma M} + \gamma J_{\gamma^{-1}M^{-1}}\circ \gamma^{-1} \id.
\end{align}
Moreover, for $f \in \Gamma(\h)$ and $\gamma \in \R_{++}$ the subdifferential $\partial (\gamma f)$ is maximally monotone (cf. \cite[Theorem 3.2.8]{Zalinescu02}) and it holds $J_{\gamma \partial f} = \left(\id + \gamma \partial f \right)^{-1} = \Prox_{\gamma f}$. Here, $\Prox_{\gamma f}(x)$ denotes the \textit{proximal point} of $\gamma f$ at $x\in\h$ representing the unique optimal solution of the optimization problem
\begin{align}\label{prox-def}
\inf_{y\in \h}\left \{\gamma f(y)+\frac{1}{2}\|y-x\|^2\right\}.
\end{align}
In this particular situation \eqref{res-identity} becomes \textit{Moreau's decomposition formula}
\begin{align}
	\label{res-indentity}
	\id = \Prox\nolimits_{\gamma f} + \gamma \Prox\nolimits_{\gamma^{-1}f^*} \circ \gamma^{-1}\id.
\end{align}
When $\Omega \subseteq \h$ is a nonempty, convex and closed set, the function $\delta_\Omega : \h \rightarrow \overline \R$, defined by $\delta_\Omega(x) = 0$ for $x \in \Omega$ and $\delta_\Omega(x) = +\infty$, otherwise, denotes the \textit{indicator function} of the set $\Omega$. For each $\gamma > 0$ the proximal point of $\gamma \delta_\Omega$ at $x \in \h$ is nothing else than
$$\Prox\nolimits_{\gamma \delta_\Omega}(x) = \Prox\nolimits_{\delta_\Omega}(x) = \proj_\Omega(x) = \argmin_{y \in \Omega} \frac{1}{2}\|y-x\|^2,$$
which is the projection of $x$ on $\Omega$.

The \textit{sum} and the \textit{parallel sum} of two set-valued operators $M_1,\,M_2:\h \rightarrow 2^{\h}$ are defined as
$M_1 + M_2:\h \rightarrow 2^{\h}, (M_1 + M_2)(x) = M_1(x) + M_2(x) \ \forall x \in \h$
and
$$M_1 \Box M_2:\h \rightarrow 2^{\h}, M_1 \Box M_2  = \left(M_1^{-1} + M_2^{-1}\right)^{-1},$$
respectively. If $M_1$ and $M_2$ are monotone, than $M_1 + M_2$ and $M_1 \Box M_2$ are monotone, too, however, if $M_1$ and $M_2$ are maximally monotone, this property is in general not true neither for $M_1 + M_2$ nor for $M_1 \Box M_2$ (see \cite{Bot10}).

\section{Algorithms and convergence results}\label{sectionInc}
Within this section we provide two algorithms together with weak and strong convergence results for the following primal-dual pair of  monotone inclusion problems.
\begin{problem}\label{p1}
Let $A:\h \rightarrow 2^{\h}$ be a maximally monotone operator and $z \in \h$. Furthermore, for every $i \in \{1,\ldots,m\}$, let $r_i \in \g_i$, $B_i : \g_i \rightarrow 2^{\g_i}$ and $D_i : \g_i \rightarrow 2^{\g_i}$ be maximally monotone operators and $L_i : \h \rightarrow \g_i$ a nonzero bounded linear operator. The problem is to solve the primal inclusion
\begin{align}
	\label{opt-p}
	\text{find }\bx \in \h \text{ such that } z \in A\bx + \sum_{i=1}^m L_i^* (B_i\Box D_i)(L_i \bx-r_i)
\end{align}
together with the dual inclusion
\begin{align}
	\label{opt-d}
	\text{find }\bv_1 \in \g_1,\ldots,\bv_m \in \g_m \text{ such that }(\exists x\in\h)\left\{
	\begin{array}{l}
		z - \sum_{i=1}^m L_i^*\bv_i \in Ax \\
		\!\!\bv_i \in \!\! (B_i \Box D_i)(L_ix-r_i), \,i=1,\ldots,m.
	\end{array}
\right.
\end{align}
\end{problem}

We say that $(\bx, \bv_1,\ldots,\bv_m) \in \h \times \g_1 \ldots \times \g_m$ is a primal-dual solution to Problem \ref{p1}, if
\begin{equation}\label{operator-proof-conditions-full}
z - \sum_{i=1}^m L_i^*\bv_i \in A\bx \ \mbox{and} \  \bv_i \in (B_i \Box D_i)(L_i\bx-r_i), \,i=1,\ldots,m.
\end{equation}
If $(\bx, \bv_1,\ldots,\bv_m) \in \h \times \g_1 \ldots \times \g_m$ is a primal-dual solution to Problem \ref{p1}, then $\bx$ is a solution to \eqref{opt-p} and $(\bv_1,\ldots,\bv_m)$ is a solution to \eqref{opt-d}. Notice also that
\begin{align*}
\bx \text{ solves }\eqref{opt-p} & \Leftrightarrow z - \sum_{i=1}^m L_i^* (B_i\Box D_i)(L_i \bx-r_i)  \in A\bx \Leftrightarrow\\
\exists\, \bv_1\in\g_1,\ldots,\bv_m\in\g_m \ & \mbox{such that} \ \left\{
		\begin{array}{l} z - \sum_{i=1}^m L_i^*\bv_i  \in A\bx, \\ \bv_i \in  (B_i\Box D_i)(L_i \bx-r_i) , \ i=1,\ldots,m.  \end{array}\right.
\end{align*}
Thus, if $\bx$ is a solution to \eqref{opt-p}, then there exists $(\bv_1,\ldots,\bv_m) \in \g_1 \times \ldots \g_m$ such that $(\bx, \bv_1,\ldots,\bv_m)$ is a primal-dual solution to Problem \ref{p1} and if $(\bv_1,\ldots,\bv_m) \in \g_1 \times \ldots \g_m$ is a solution to \eqref{opt-d}, then there exists $\bx \in \h$ such that $(\bx, \bv_1,\ldots,\bv_m)$ is a primal-dual solution to Problem \ref{p1}.

\begin{example}\label{ex1}
In Problem \ref{p1}, set $m=1$, $z=0$ and $r_1=0$, let $\g_1=\g$, $B_1=B$, $L_1=L$, $D_1 : \g \rightarrow 2^{\g}$, $D_1(0) = \g$ and $D_1(v) = \varnothing \ \forall v \in \g \setminus \{0\}$, and $A:\h \rightarrow 2^{\h} $ and $B:\g \rightarrow 2^{\g}$ be the convex subdifferentials of the functions $f \in \Gamma(\h)$ and $g\in\Gamma(\g)$, respectively. Then, under appropriate qualification conditions (see \cite{Bot10,BGW09}), to solve the primal inclusion problem \eqref{opt-p} is equivalent to solve the optimization problem
\begin{align*}
		 \inf_{x\in\h} \left\{f(x) + g(Lx)\right\},
\end{align*}		
while to solve the dual inclusion problem \eqref{opt-d} is nothing else than to solve its Fenchel-dual problem
\begin{align*}
		\sup_{v \in \g}\left\{ - f^*(-L^*v) - g^*(v)\right\}.
\end{align*}
\end{example}
For more primal-dual pairs of convex optimization problems which are particular instances of \eqref{opt-p}-\eqref{opt-d} we refer to \cite{ComPes12,Vu11}.

\subsection{A first primal-dual algorithm}\label{subsectionInc1}
The first iterative scheme we propose in this paper and which we describe as follows has the particularity that it accesses the resolvents of $A$, $B_i^{-1}$ and $D_i^{-1}, i=1,...,m,$ and processes each operator $L_i$ and its adjoint $L_i^*, i=1,...,m$ two times.

\begin{algorithm}\label{alg1} \text{ }\newline
Let $x_0 \in \h$, $(v_{1,0}, \ldots, v_{m,0}) \in \g_1 \times \ldots \times \g_m$ and $\tau$ and $\sigma_i, i=1,...,m,$ be strictly positive real numbers such that
$$\tau \sum_{i=1}^m \sigma_i \|L_i\|^2 < 4. $$
Furthermore, let $(\lambda_n)_{n\geq 0}$ be a sequence in $(0,2)$, $(a_n)_{n\geq 0}$ a sequence in $\h$, $(b_{i,n})_{n\geq 0}$ and $(d_{i,n})_{n\geq 0}$ sequences in $\g_i$ for all $i=1,\ldots,m$ and set
	\begin{align}\label{A1}
	  \left(\forall n\geq 0\right) \begin{array}{l} \left\lfloor \begin{array}{l}
		p_{1,n} = J_{\tau A}\left( x_n - \frac{\tau}{2} \sum_{i=1}^m L_i^* v_{i,n} + \tau z \right) + a_n \\
		w_{1,n} = 2p_{1,n} - x_n \\
		\text{For }i=1,\ldots,m  \\
				\left\lfloor \begin{array}{l}
					p_{2,i,n} = J_{\sigma_i B_i^{-1}}\left(v_{i,n} +\frac{\sigma_i}{2} L_i w_{1,n} - \sigma_i r_i \right) + b_{i,n} \\
					w_{2,i,n} = 2 p_{2,i,n} - v_{i,n}
				\end{array} \right.\\
		z_{1,n} = w_{1,n} - \frac{\tau}{2} \sum_{i=1}^m L_i^* w_{2,i,n} \\
		x_{n+1} = x_n + \lambda_n ( z_{1,n} - p_{1,n} ) \\
		\text{For }i=1,\ldots,m  \\
				\left\lfloor \begin{array}{l}
					z_{2,i,n} = J_{\sigma_i D_i^{-1}}\left(w_{2,i,n} + \frac{\sigma_i}{2}L_i (2 z_{1,n} - w_{1,n}) \right) + d_{i,n} \\
					v_{i,n+1} = v_{i,n} + \lambda_n (z_{2,i,n} - p_{2,i,n}).
				\end{array} \right. \\ \vspace{-4mm}
		\end{array}
		\right.
		\end{array}
	\end{align}
\end{algorithm}

\begin{theorem}\label{thm1}
For Problem \ref{p1} assume that
\begin{align}\label{zin}
	z \in \ran \left( A +  \sum_{i=1}^m L_i^*(B_i\Box D_i)(L_i \cdot -r_i) \right)
\end{align}
and consider the sequences generated by Algorithm \ref{alg1}.
\begin{enumerate}[label={(\roman*)}]
	\setlength{\itemsep}{-2pt}
		\item  \label{th1.01} If
$$\sum_{n=0}^{+\infty}\lambda_n\|a_n\|_{\h} < +\infty, \quad \sum_{n=0}^{+\infty}\lambda_n (\|d_{i,n}\|_{\g_i} + \|b_{i,n}\|_{\g_i} ) < + \infty, i=1,\ldots,m,$$ and
$\sum_{n=0}^{+\infty} \lambda_n (2 - \lambda_n) = +\infty,$
then
\begin{enumerate}[label={(\alph*)}]
	\setlength{\itemsep}{-2pt}
		\item  \label{th1.1}$(x_n,v_{1,n},\ldots,v_{m,n})_{n\geq 0}$ converges weakly to a point $(\bx,\bv_1,\ldots,\bv_m)\in\h\times\g_1\times\ldots\times\g_m$ such that, when setting
		\begin{align*}
						\bp_1 &= J_{\tau A}\left(\bx - \frac{\tau}{2} \sum_{i=1}^m L_i^*\bv_i + \tau z \right), \\
						\text{and }\bp_{2,i} & =J_{\sigma_i B_i^{-1}}\left(\bv_{i} +\frac{\sigma_i}{2} L_i (2\bp_1-\bx) -\sigma_i r_i \right), \ i =1,...,m,
\end{align*}
the element $(\bp_1, \bp_{2,1},\ldots,\bp_{2,m})$ is a primal-dual solution to Problem \ref{p1}.
		\item  \label{th1.2}$\lambda_n (z_{1,n}-p_{1,n}) \rightarrow 0 \  (n \rightarrow +\infty)$ and  $\lambda_n(z_{2,i,n}-p_{2,i,n}) \rightarrow  0 \  (n \rightarrow +\infty)$ for $i=1,...,m$.
		\item \label{th1.3} whenever $\h$ and $\g_i$, $i=1,...,m,$ are finite-dimensional Hilbert spaces, $a_n \rightarrow 0 \  (n \rightarrow +\infty)$ and $b_{i,n} \rightarrow 0 \  (n \rightarrow +\infty)$ for $i=1,...,m$, then $(p_{1,n}, p_{2,1,n},\ldots,p_{2,m,n})_{n \geq 0}$ converges strongly to a primal-dual solution of Problem \ref{p1}.
		\end{enumerate}
\item  \label{th1.02} If
$$\sum_{n=0}^{+\infty}\|a_n\|_{\h} < +\infty, \quad \sum_{n=0}^{+\infty} (\|d_{i,n}\|_{\g_i} + \|b_{i,n}\|_{g_i} ) < + \infty, i=1,\ldots,m, \quad \inf_{n\geq 0} \lambda_n > 0$$
$$\mbox{and} \ A \ \mbox{and} \ B_i^{-1}, i=1,...,m, \ \mbox{are uniformly monotone},$$
then $(p_{1,n}, p_{2,1,n},\ldots,p_{2,m,n})_{n \geq 0}$ converges strongly to the unique primal-dual solution of Problem \ref{p1}.
	\end{enumerate}
\end{theorem}

\begin{proof}
Consider the Hilbert space $\fG = \g_1 \times \ldots \times \g_m$ endowed with inner product and associated norm defined, for  $\fv=(v_1,\ldots,v_m)$, $\fq=(q_1,\ldots,q_m) \in \fG$, as
\begin{align}
	\label{def1.01}
	\< \fv,\fq \>_{\fG} = \sum_{i=1}^m \< v_i,q_i \>_{\g_i}  \text{ and } \|\fv\|_{\fG} = \sqrt{\sum_{i=1}^m \| v_i \|_{\g_i}^2},
\end{align}
respectively. Furthermore, consider the Hilbert space $\fK = \h \times \fG$ endowed with inner product and associated norm defined, for $(x,\fv),(y,\fq) \in \fK$, as
\begin{align}
	\< (x,\fv),(y,\fq) \>_{\fK} = \< x,y \>_{\h} + \< \fv,\fq \>_{\fG}  \text{ and } \|(x,\fv)\|_{\fK} = \sqrt{ \| x \|_{\h}^2 + \| \fv \|_{\fG}^2},
\end{align}
respectively. Consider the set-valued operator
\begin{align*}
	\f M : \fK \rightarrow 2^{\fK}, \quad (x,v_1,\ldots,v_m) \mapsto (-z + Ax, r_1 + B_1^{-1}v_1, \ldots, r_m + B_m^{-1}v_m),
\end{align*}
which is maximally monotone,  since $A$ and $B_i, i=1,...,m,$ are maximally monotone (cf. \cite[Proposition 20.22 and Proposition 20.23]{BauCom11}) and the bounded linear operator
\begin{align*}
	\f S : \fK \rightarrow \fK, \quad (x,v_1,\ldots,v_m) &\mapsto \left(\sum_{i=1}^m L_i^* v_i, -L_1 x , \ldots, - L_m x\right),
\end{align*}
which proves to be skew (i.\,e. $\f S^*=-\f S$) and hence maximally monotone (cf. \cite[Example 20.30]{BauCom11}). Further, consider the set-valued operator
\begin{align*}
	\f Q : \fK \rightarrow 2^{\fK}, \quad (x,v_1,\ldots,v_m) \mapsto \left(0, D_1^{-1} v_1, \ldots, D_m^{-1} v_m\right),
\end{align*}
which is maximally monotone, as well, since $D_i$ is maximally monotone for $i=1,...,m$. Therefore, since $\dom \f S = \fK$, both $\frac{1}{2}\f S+ \f Q$ and $\frac{1}{2}\f S + \f M$ are maximally monotone (cf. \cite[Corollary 24.4(i)]{BauCom11}). On the other hand, according to \cite[Eq. (3.12)]{ComPes12}, it holds \eqref{zin} $\Leftrightarrow \,\zer\left(\f M + \f S + \f Q\right) \neq \varnothing$, while \cite[Eq. (3.21) and (3.22)]{ComPes12} yield
\begin{align}\label{optcond}
	(x,v_1,\ldots,v_m) \in \zer\left(\f M + \f S + \f Q\right) \Rightarrow & (x, v_1,\ldots,v_m) \text{ is a primal-dual}\notag \\
& \mbox{solution to Problem } \ref{p1}.
\end{align}
Finally, we introduce the bounded linear operator
\begin{align*}
	\f V : \fK &\rightarrow \fK, \quad (x,v_1,\ldots,v_m) &\mapsto \left(\frac{x}{\tau} -\frac{1}{2} \sum_{i=1}^m L_i^* v_i, \frac{v_1}{\sigma_1} - \frac{1}{2} L_1x , \ldots, \frac{v_m}{\sigma_m} -\frac{1}{2} L_m x\right).
\end{align*}
It is a simple calculation to prove that $\f V$ is self-adjoint, i.\,e. $\f V^* = \f V$. Furthermore, the operator $\f V$ is $\rho$-strongly positive for
$$\rho = \left(1-\frac{1}{2} \sqrt{\tau \sum_{i=1}^m \sigma_i \|L_i\|^2}\right) \min\left\{\frac{1}{\tau}, \frac{1}{\sigma_1}, \ldots, \frac{1}{\sigma_m} \right\},$$
which is a positive real number due to the assumption
\begin{align}
	\tau \sum_{i=1}^m \sigma_i \|L_i\|^2 < 4
\end{align}
made in Algorithm \ref{alg1}. Indeed, using that $2ab \leq \alpha a^2 + \frac{b^2}{\alpha}$ for any $a,\,b\in \R$ and any $\alpha \in \R_{++}$, it yields for each $i=1,\ldots,m$
\begin{align}
	\label{ineq1}
	2\| L_i \| \|x\|_{\h} \|v_i\|_{\g_i} \leq  \frac{\sigma_i \|L_i\|^2}{\sqrt{\tau \sum_{i=1}^m \sigma_i \|L_i\|^2}} \|x\|_{\h}^2 + \frac{\sqrt{\tau \sum_{i=1}^m \sigma_i \|L_i\|^2}}{\sigma_i} \|v_i\|_{\g_i}^2
\end{align}
and, consequently, for each $\fx=(x,v_1,\ldots,v_m) \in \fK$, it follows that
\begin{align}
	\<\fx, \f V\fx \>_{\fK} &= \frac{\|x\|_{\h}^2}{\tau} + \sum_{i=1}^m \frac{\|v_i\|_{\g_i}^2}{\sigma_i} - \sum_{i=1}^m \< L_i x, v_i\>_{\g_i} \notag \\
	&\geq \frac{\|x\|_{\h}^2}{\tau} + \sum_{i=1}^m \frac{\|v_i\|_{\g_i}^2}{\sigma_i} - \sum_{i=1}^m \|L_i\| \|x\|_{\h} \|v_i\|_{\g_i} \notag \\
	&\overset{\mathclap{\eqref{ineq1}}}{\geq} \left(1-\frac{1}{2} \sqrt{\tau \sum_{i=1}^m \sigma_i \|L_i\|^2}\right) \left( \frac{\|x\|_{\h}^2}{\tau} + \sum_{i=1}^m \frac{\|v_i\|_{\g_i}^2}{\sigma_i} \right)  \notag \\
	&\geq \left(1-\frac{1}{2} \sqrt{\tau \sum_{i=1}^m \sigma_i \|L_i\|^2}\right) \min\left\{\frac{1}{\tau}, \frac{1}{\sigma_1}, \ldots, \frac{1}{\sigma_m} \right\} \|\fx\|_{\fK}^2 \notag \\
	&= \rho  \|\fx\|_{\fK}^2.
\end{align}
Since $\f V$ is $\rho$-strongly positive, we have $\cl(\ran \f V)=\ran \f V$ (cf. \cite[Fact 2.19]{BauCom11}), $\zer \f V = \{0\}$ and, as $(\ran \f V)^{\bot} = \zer \f V^* = \zer \f V = \{0\}$ (see, for instance, \cite[Fact 2.18]{BauCom11}), it holds $\ran \f V = \fK$. Consequently, $\f V^{-1}$ exists and $\|\f V^{-1}\|\leq \frac{1}{\rho}$.

The algorithmic scheme \eqref{A1} is equivalent to
	\begin{align}\label{A1.1}
	  \left(\forall n\geq 0\right) \begin{array}{l} \left\lfloor \begin{array}{l}
		\frac{x_n- p_{1,n}}{\tau} - \frac{1}{2} \sum_{i=1}^m L_i^* v_{i,n} \in A(p_{1,n} -a_n) -z -\frac{a_n}{\tau} \\
		w_{1,n} = 2p_{1,n} - x_n \\
		\text{For }i=1,\ldots,m  \\
				\left\lfloor \begin{array}{l}
				 \frac{v_{i,n} - p_{2,i,n}}{\sigma_i}	- \frac{1}{2} L_i (x_n - p_{1,n}) \in -\frac{1}{2}L_i p_{1,n} + B_i^{-1}(p_{2,i,n} - b_{i,n}) +r_i- \frac{b_{i,n}}{\sigma_i} \\
					w_{2,i,n} = 2 p_{2,i,n} - v_{i,n} \\
				\end{array} \right.\\
		\frac{w_{1,n}-z_{1,n}}{\tau} - \frac{1}{2}\sum_{i=1}^m L_i^* w_{2,i,n} = 0 \\
		x_{n+1} = x_n + \lambda_n ( z_{1,n} - p_{1,n} ) \\
		\text{For }i=1,\ldots,m  \\
				\left\lfloor \begin{array}{l}
					\frac{w_{2,i,n} - z_{2,i,n}}{\sigma_i} - \frac{1}{2}L_i(w_{1,n}-z_{1,n}) \in -\frac{1}{2}L_i z_{1,n} + D_i^{-1}(z_{2,i,n}-d_{i,n}) - \frac{d_{i,n}}{\sigma_i} \\
					v_{i,n+1} = v_{i,n} + \lambda_n (z_{2,i,n} - p_{2,i,n}). \\
				\end{array} \right. \\		
		\end{array}
		\right.
		\end{array}
	\end{align}
We introduce for every $n \geq 0$ the following notations:
\begin{align}\label{seq0}
\left\{
	\begin{array}{l}
	\fx_n = (x_n,v_{1,n},\ldots,v_{m,n}) \\
	\fy_n =(p_{1,n},p_{2,1,n},\ldots,p_{2,m,n}) \\
	\fw_n =(w_{1,n},w_{2,1,n},\ldots,w_{2,m,n}) \\
	\fz_n =(z_{1,n},z_{2,1,n},\ldots,z_{2,m,n})
	\end{array}\right.
	\text{ and }
\left\{
	\begin{array}{l}
	\f d_n = (0,d_{1,n},\ldots,d_{m,n}) \\
	\f d_n^{\sigma} =(0,\frac{d_{1,n}}{\sigma_1},\ldots,\frac{d_{m,n}}{\sigma_m}) \\
	\f b_n =(a_n,b_{1,n},\ldots,b_{m,n}) \\
	\f b_n^{\sigma} =(\frac{a_n}{\tau},\frac{b_{1,n}}{\sigma_1},\ldots,\frac{b_{m,n}}{\sigma_m})
	\end{array}\right..
\end{align}

The scheme \eqref{A1.1} can equivalently be written in the form
\begin{align}
	\label{A1.2}
	\left(\forall n\geq 0\right)  \left\lfloor \begin{array}{l}
	\f V(\fx_n - \fy_n) \in \left(\frac{1}{2}\f S +\f M\right)\left(\fy_n - \f b_n\right) + \frac{1}{2} \f S \f b_n- \f b_n^{\sigma} \\
	\fw_n = 2 \fy_n - \fx_n \\
	\f V(\fw_n - \fz_n) \in \left(\frac{1}{2}\f S +\f Q\right)\left(\fz_n - \f d_n\right) + \frac{1}{2} \f S \f d_n- \f d_n^{\sigma} \\
	\fx_{n+1} = \fx_n + \lambda_n \left(\fz_n-\fy_n\right).	
	\end{array}
	\right.
\end{align}

We set for every $n \geq 0$
\begin{align}\label{seq1}
\begin{aligned}
	\f e_n^b&=\f V^{-1}\left( \left(\frac{1}{2}\f S + \f V\right)\f b_n - \f b_n^{\sigma} \right) \\
	\f e_n^d&=\f V^{-1}\left( \left(\frac{1}{2}\f S + \f V\right)\f d_n - \f d_n^{\sigma} \right).
\end{aligned}
\end{align}

Next we introduce the Hilbert space $\fK_{\f V}$ with inner product and norm respectively defined, for $\fx,\fy \in \fK$, as
\begin{align}\label{HSKV}
 \< \fx,\fy \>_{\fK_{\f V}} = \< \fx, \f V \fy \>_{\fK} \text{ and } \|\fx\|_{\fK_{\f V}} = \sqrt{\< \fx, \f V \fx \>_{\fK}},
\end{align}
respectively. Since $\frac{1}{2}\f S+ \f M$ and $\frac{1}{2}\f S+ \f Q$ are maximally monotone on $\fK$, the operators
\begin{align}\label{def1.1}
		\f B := \f V^{-1}\left(\frac{1}{2}\f S +\f M\right) \ \mbox{and} \ \f A:=\f V^{-1}\left(\frac{1}{2}\f S +\f Q\right)
\end{align}
are maximally monotone on $\fK_{\f V}$. Moreover, since $\f V$ is self-adjoint and $\rho$-strongly positive, one can easily see that weak and strong convergence in $\fK_{\f V}$ are equivalent with weak and strong convergence in $\fK$, respectively.

Consequently, for every $n \geq 0$ we have
\begin{align}\label{inc1.1}
	& \f V(\fx_n - \fy_n) \in \left(\frac{1}{2}\f S +\f M\right)\left(\fy_n - \f b_n\right) + \frac{1}{2} \f S \f b_n- \f b_n^{\sigma}  \notag \\
	\Leftrightarrow \ & \f V\fx_n \in \left(\f V + \frac{1}{2}\f S +\f M\right)\left(\fy_n - \f b_n\right) + \left(\frac{1}{2} \f S + \f V\right)\f b_n- \f b_n^{\sigma} \notag \\
	\Leftrightarrow \ & \fx_n \in \left(\id + \f V^{-1}\left(\frac{1}{2}\f S +\f M\right)\right)\left(\fy_n - \f b_n\right) + \f V^{-1}\left(\left(\frac{1}{2} \f S + \f V\right)\f b_n- \f b_n^{\sigma}\right) \notag \\
	\Leftrightarrow \ & \fy_n = \left(\id + \f V^{-1}\left(\frac{1}{2}\f S +\f M\right)\right)^{-1}\left(\fx_n - \f e_n^b\right) + \f b_n \notag \\
\Leftrightarrow \ & \fy_n = \left(\id + \f B\right)^{-1}\left(\fx_n - \f e_n^b\right) + \f b_n
\end{align}
and
\begin{align}\label{inc1.2}
	& \f V(\fw_n - \fz_n) \in \left(\frac{1}{2}\f S +\f Q\right)\left(\fz_n - \f d_n\right) + \frac{1}{2} \f S \f d_n- \f d_n^{\sigma} \notag \\
	\Leftrightarrow \ & \fz_n = \left(\id + \f V^{-1}\left(\frac{1}{2}\f S +\f Q\right)\right)^{-1}\left(\fw_n - \f e_n^d\right) + \f d_n \notag \\
\Leftrightarrow \ & \fz_n = \left(\id + \f A\right)^{-1}\left(\fw_n - \f e_n^d\right) + \f d_n.
\end{align}
Thus, the iterative rules in \eqref{A1.2} become
\begin{align}
	\label{A1.3}
	\left(\forall n\geq 0\right)  \left\lfloor \begin{array}{l}
	\fy_n = J_{\f B}\left(\fx_n - \f e_n^b\right) + \f b_n \\
	\fz_n = J_{\f A}\left(2\fy_n - \fx_n - \f e_n^d\right) + \f d_n \\
	\fx_{n+1} = \fx_n + \lambda_n(\fz_n-\fy_n)	
	\end{array}
	\right..
\end{align}

In addition, we have
\begin{align*}
	\zer \left( \f A + \f B\right) = \zer \left( \f V^{-1}\left( \f M + \f S + \f Q\right) \right) = \zer \left( \f M + \f S + \f Q \right).
\end{align*}

By defining for every $n \geq 0$
$$\f \beta_n = J_{\f B}\left(\fx_n - \f e_n^b\right) - J_{\f B}\left(\fx_n\right) + \f b_n \ \mbox{and} \ \f \alpha_n =  J_{\f A}\left(2\fy_n - \fx_n - \f e_n^d\right)  - J_{\f A}\left(2\fy_n - \fx_n\right) + \f d_n,$$
the iterative scheme \eqref{A1.3} becomes
\begin{align}
	\label{A1.5}
	\left(\forall n\geq 0\right)  \left\lfloor \begin{array}{l}
	\fy_n = J_{\f B}(\fx_n) + \f \beta_n \\
	\fz_n = J_{\f A}\left(2\fy_n - \fx_n\right) + \f \alpha_n \\
	\fx_{n+1} = \fx_n + \lambda_n(\fz_n-\fy_n)	
	\end{array}
	\right.,
\end{align}
thus, it has the structure of an error-tolerant Douglas-Rachford algorithm (see \cite{Com09}).

\ref{th1.01} The assumptions made on the error sequences yield
\begin{align}
	\label{sum1.1}
	\sum_{n=0}^{+\infty} \lambda_n \|\f d_n\|_{\fK} < +\infty, \ \sum_{n=0}^{+\infty} \lambda_n\|\f d_n^{\sigma}\|_{\fK} < +\infty,
	\sum_{n=0}^{+\infty} \lambda_n \|\f b_n\|_{\fK} < +\infty, \ \sum_{n=0}^{+\infty} \lambda_n \|\f b_n^{\sigma}\|_{\fK} < +\infty
\end{align}
and, by the boundedness of $\f V^{-1}$, $\f S$ and $\f V$, if follows
\begin{align}
\label{sum1.2}
\sum_{n=0}^{+\infty} \lambda_n \|\f e_n^b\|_{\fK} < +\infty \text{ and } \sum_{n=0}^{+\infty} \lambda_n \|\f e_n^d \|_{\fK} < +\infty.
\end{align}
Further, by making use of the nonexpansiveness of the resolvents, the error sequences satisfy
\begin{align*}
\sum_{n=0}^{+\infty} \lambda_n \left [\|\f \alpha_n\|_{\fK} + \|\f \beta_n\|_{\fK} \right] \leq & \sum_{n=0}^{+\infty} \lambda_n \left[\|J_{\f A}\left(2\fy_n - \fx_n - \f e_n^d\right)  - J_{\f A}\left(2\fy_n - \fx_n\right)\|_{\fK} + \|\f d_n\|_{\fK}\right.\\
& \quad \qquad \left. + \|J_{\f B}\left(\fx_n - \f e_n^b\right) - J_{\f B}\left(\fx_n\right)\|_{\fK} + \|\f b_n\|_{\fK} \right]\\
\leq & \sum_{n=0}^{+\infty} \lambda_n \left [\|\f e_n^d\|_{\fK} + \|\f d_n\|_{\fK} + \|\f e_n^b\|_{\fK} + \|\f b_n\|_{\fK} \right] < + \infty.
\end{align*}

By the linearity and boundedness of $\f V$ it follows that
\begin{align*}
	\sum_{n=0}^{+\infty} \lambda_n \left [\|\f \alpha_n\|_{\fK_{\f V}} + \|\f \beta_n\|_{\fK_{\f V}} \right]  < +\infty.
\end{align*}

\ref{th1.01}\ref{th1.1} According to \cite[Theorem 2.1(i)(a)]{Com09} the sequence $(\fx_n)_{n\geq 0}$ converges weakly in $\fK_{\f V}$ and, consequently, in $\fK$ to a point $\fbx \in \Fix\left(R_{\f A} R_{\f B}\right)$ with $J_{\f B}\fbx \in \zer(\f A + \f B)$. The claim follows by identifying $J_{\f B}\fbx$ and by noting \eqref{optcond}.

\ref{th1.01}\ref{th1.2} According to \cite[Theorem 2.1(i)(b)]{Com09} it follows that $(R_{\f A} R_{\f B} \fx_n - \fx_n) \rightarrow 0 \ (n \rightarrow +\infty)$. From \eqref{A1.5} it follows that for every $n \geq 0$
$$\lambda_n (\fz_n-\fy_n) = \frac{\lambda_n}{2} \left (R_{\f A}(R_{\f B}(\fx_n) + 2\f \beta_n) - \fx_n + 2\f \alpha_n  \right),$$
thus, by taking into consideration the nonexpansiveness of the reflected resolvent and the boundedness of $(\lambda_n)_{n \geq 0}$, it yields
\begin{align*}
\|\lambda_n (\fz_n-\fy_n)\|_{\fK_{\f V}} \leq & \frac{\lambda_n}{2} \|R_{\f A} R_{\f B} \fx_n - \fx_n\|_{\fK_{\f V}}\\
& + \frac{\lambda_n}{2} \|R_{\f A}(R_{\f B}\fx_n + 2\f \beta_n)-R_{\f A} (R_{\f B} \fx_n) + 2\f \alpha_n\|_{\fK_{\f V}}\\
\leq & \|R_{\f A} R_{\f B} \fx_n - \fx_n\|_{\fK_{\f V}} + \lambda_n \left [\|\f \alpha_n\|_{\fK_{\f V}} + \|\f \beta_n\|_{\fK_{\f V}} \right].
\end{align*}
The claim follows by taking into account that $\lambda_n \left [\|\f \alpha_n\|_{\fK_{\f V}} + \|\f \beta_n\|_{\fK_{\f V}} \right] \rightarrow 0 \ (n \rightarrow +\infty)$.

\ref{th1.01}\ref{th1.3} As shown in \ref{th1.1}, we have that $\fx_n \rightarrow \fbx \in \Fix\left(R_{\f A} R_{\f B}\right) \ (n \rightarrow +\infty)$ with $J_{\f B}\fbx \in \zer(\f A + \f B) = \zer(\f M + \f S + \f Q)$. Moreover, by the assumptions we have $\f b_n \rightarrow 0 \ (n \rightarrow +\infty)$ (cf. \eqref{seq0}, thus $\f e_n^b \rightarrow 0 \ (n \rightarrow +\infty)$ (cf. \eqref{seq1}) and $\f \beta_n \rightarrow 0 \ (n \rightarrow +\infty)$. Hence, by the continuity of $J_{\f B}$ and \eqref{A1.5}, we have
$$\fy_n = J_{\f B}\left(\fx_n\right) + \f \beta_n \rightarrow J_{\f B} \fbx \in \zer\left( \f M + \f S + \f Q \right) \ (n \rightarrow +\infty).$$

\ref{th1.02} The assumptions made on the error sequences yield
\begin{align*}
	\sum_{n=0}^{+\infty} \|\f d_n\|_{\fK} < +\infty, \ \sum_{n=0}^{+\infty} \|\f d_n^{\sigma}\|_{\fK} < +\infty,
	\sum_{n=0}^{+\infty} \|\f b_n\|_{\fK} < +\infty, \ \sum_{n=0}^{+\infty} \|\f b_n^{\sigma}\|_{\fK} < +\infty,
\end{align*}
thus,
\begin{align*}
\sum_{n=0}^{+\infty} \|\f e_n^b\|_{\fK} < +\infty \text{ and } \sum_{n=0}^{+\infty}\|\f e_n^d \|_{\fK} < +\infty.
\end{align*}
This implies that
\begin{align*}
\sum_{n=0}^{+\infty} \left [\|\f \alpha_n\|_{\fK} + \|\f \beta_n\|_{\fK} \right] < + \infty
\end{align*}
which, due to the linearity and boundedness of $\f V$, yields
\begin{align*}
	\sum_{n=0}^{+\infty} \left [\|\f \alpha_n\|_{\fK_{\f V}} + \|\f \beta_n\|_{\fK_{\f V}} \right]  < +\infty.
\end{align*}
Since $A$ and $B_i^{-1}, i=1,...,m,$ are uniformly monotone, there exist increasing functions $\phi_{A}: \R_+ \rightarrow [0,+\infty]$ and  $\phi_{B_i^{-1}}: \R_+ \rightarrow [0,+\infty], i=1,...,m$, vanishing only at $0$, such that
\begin{align}
		\begin{aligned}
				 \<x-y,u-z\> &\geq \phi_A \left( \| x-y \|_{\h} \right)  \ \forall\,(x,u),(y,z) \in \gra A  \\
				 \<v-w,p-q\> &\geq \phi_{B_i^{-1}} \left( \| v-w \|_{\g_i} \right)  \ \forall\,(v,p),(w,q) \in \gra B_i^{-1} \ \forall i=1,...,m.
		\end{aligned}
\end{align}
 The function $\phi_{\f M}: \R_+ \rightarrow [0,+\infty]$,
\begin{align}
	\label{def1.2}
	\phi_{\f M}(c) = \inf\left\{ \phi_{A}(a) + \sum_{i=1}^m\phi_{B_i^{-1}}(b_i): \sqrt{a^2 + \sum_{i=1}^m{b_i^2}}=c \right\},
\end{align}
is increasing and vanishes only at $0$ and it fulfills for each $(\fx,\f u), (\fy, \fz) \in \gra \f M$
\begin{align}
		\label{ineq1.1}
		\<\fx-\fy, \f u-\fz \>_{\fK} \geq \phi_{\f M}\left( \|\fx-\fy\|_{\fK} \right).
\end{align}
Thus, $\f M$ is uniformly monotone on $\fK$.

The function $\phi_{\f B}: \R_+ \rightarrow [0,+\infty]$, $\phi_{\f B}(t)=\phi_{\f M}\left(\frac{1}{\sqrt{\|\f V\|}}t\right)$, is increasing and vanishes only at $0$. Let be $(\fx,\f u), (\fy,\fz) \in \gra \f B$. Then there exist $\f v \in \f M \fx$ and  $\f w \in \f M\fy$ fulfilling $\f V \f u = \frac{1}{2}\f S \fx + \f v$ and $\f V \fz = \frac{1}{2}\f S \fy + \f w$ and it holds
\begin{align}\label{ineq1.2}
 \< \fx-\fy, \f u - \fz \>_{\fK_{\f V}}
&=	\< \fx-\fy, \f V \f u - \f V \fz \>_{\fK} \notag \\
&=  \< \fx-\fy, \left(\frac{1}{2}\f S \f x+ \f v\right)- \left(\frac{1}{2} \f S \f y+ \f w \right) \>_{\fK} \notag \\
&\overset{\mathclap{\eqref{ineq1.1}}}{\geq} \phi_{\f M}\left( \| \fx - \fy \|_{\fK} \right) \notag \\
&\geq \phi_{\f M}\left(\frac{1}{\sqrt{\|\f V\|}}\|\fx - \fy \|_{\fK_{\f V}} \right) \notag \\
&= \phi_{\f B}\left( \| \fx - \fy \|_{\fK_{\f V}} \right).
\end{align}
Consequently, $\f B$ is uniformly monotone on $\fK_{\f V}$ and, according to \cite[Theorem 2.1(ii)(b)]{Com09}, $(J_{\f B}\fx_n)_{n\geq 0}$ converges strongly to the unique element $\fby \in \zer(\f A + \f B) = \zer\left(\f M + \f S+ \f Q\right)$. In the light of \eqref{A1.5} and using that $\f \beta_n \rightarrow 0 \ (n \rightarrow +\infty)$, it follows that $\fy_n \rightarrow \fby \ (n \rightarrow +\infty)$.
\end{proof}

\begin{remark}\label{rm1.1} Some remarks concerning Algorithm \ref{alg1} and Theorem \ref{thm1} are in order.
\begin{enumerate}[label={(\roman*)}]
	\setlength{\itemsep}{-2pt}
		\item \label{rm1.1.1} Algorithm \ref{alg1} is a fully decomposable iterative method, as each of the operators occurring in Problem \ref{p1} is processed individually. Moreover, a considerable number of steps in \eqref{A1} can be executed in parallel.
		\item \label{rm1.1.2} The proof of Theorem \ref{thm1}, which states the convergence of Algorithm \ref{alg1}, relies on the reformulation of the iterative scheme as an inexact Douglas-Rachford method in a specific real Hilbert space. For the use of a similar technique in the context of a forward-backward-type method we refer to \cite{Vu11}.
		\item \label{rm1.1.3} We would like to notice that the assumption $\sum_{n=0}^{+\infty}\lambda_n\|a_n\|_{\h} < +\infty$ does not necessarily imply neither that $(\|a_n\|_{\h})_{n \geq 0}$ is summable nor that $(a_n)_{n \geq 0}$ (weakly or strongly) converges to $0$ as $n \rightarrow +\infty$. We refer to \cite[Remark 2.2(iii)]{Com09} for further considerations on the conditions imposed on the error sequences in Theorem \ref{thm1}.
\end{enumerate}
\end{remark}

\begin{remark}\label{rm1.2} In the following we emphasize the relations between the proposed algorithm and other existent primal-dual iterative schemes.
\begin{enumerate}[label={(\roman*)}]
		\setlength{\itemsep}{-2pt}
		\item \label{rm1.2.1} Other iterative methods for solving the primal-dual monotone inclusion pair introduced in Problem \ref{p1} were given in
\cite{ComPes12} and \cite{Vu11} for $D_i^{-1}, i=1,...,m$ monotone Lipschitzian and cocoercive operators, respectively. Different to the approach proposed in this subsection, there, the operators $D_i^{-1}, i=1,...,m$, are processed within some forward steps.
		\item \label{rm1.2.2} When for every $i=1,...,m$ one takes $D_i(0) = \g_i$ and $D_i(v) = \varnothing \ \forall v \in \g_i \setminus \{0\}$, the algorithms proposed in \cite[Theorem 3.1]{ComPes12}  (see, also, \cite[Theorem 3.1]{BriCom11} for the case $m=1$) and \cite[Theorem 3.1]{Vu11} applied to Problem \ref{p1} differ from Algorithm \ref{alg1}.
\item \label{rm1.2.4} When solving the particular case of a primal-dual pair of convex optimization problems discussed in Example \ref{ex1} one can make use of the iterative schemes provided in \cite[Algorithm 3.1]{Con12}  and \cite[Algorithm 1]{ChaPoc11}. Let us notice that particularizing Algorithm \ref{alg1} to this framework gives rise to a numerical scheme different to the ones in the mentioned literature.
\end{enumerate}
\end{remark}

\subsection{A second primal-dual algorithm}\label{subsectionInc2}
In Algorithm \ref{alg1} each operator $L_i$ and its adjoint $L_i^*, i=1,...,m$ are processed two times, however, for large-scale optimization problems these matrix-vector multiplications may be expensive compared with the computation of the resolvents of the operators $A$, $B_i^{-1}$ and $D_i^{-1}, i=1,...,m$.

The second primal-dual algorithm we propose for solving  the monotone inclusions in Problem \ref{p1} has the particularity that it evaluates each operator $L_i$ and its adjoint $L_i^*, i=1,...,m$, only once.

\begin{algorithm}\label{alg2} \text{ }\newline
	Let $x_0 \in \h$, $(y_{1,0}, \ldots, y_{m,0}) \in \g_1 \times \ldots \times \g_m$, $(v_{1,0}, \ldots, v_{m,0}) \in \g_1 \times \ldots \times \g_m$, and $\tau$ and $\sigma_i, i=1,...,m,$ be strictly positive real numbers such that
$$	\tau \sum_{i=1}^m \sigma_i \|L_i\|^2 < \frac{1}{4}.$$
Furthermore, set $\gamma_i=\sigma_i^{-1}\tau\sum_{i=1}^m \sigma_i\|L_i\|^2, i=1,...,m$, let $(\lambda_n)_{n\geq 0}$ be a sequence in $(0,2)$, $(a_n)_{n\geq 0}$ a sequence in $\h$, $(b_{i,n})_{n\geq 0}$ and $(d_{i,n})_{n\geq 0}$ sequences in $\g_i$ for all $i=1,\ldots,m$ and set
\begin{align}\label{A2}
	  \left(\forall n\geq 0\right) \begin{array}{l} \left\lfloor \begin{array}{l}
		p_{1,n} = J_{\tau A}\left( x_n - \tau\left( \sum_{i=1}^m L_i^* v_{i,n} - z \right)\right) + a_n \\
		x_{n+1} = x_n + \lambda_n ( p_{1,n} - x_n ) \\
		\text{For }i=1,\ldots,m  \\
		\left\lfloor \begin{array}{l}
					p_{2,i,n} = J_{\gamma_i D_i}\left(y_{i,n} + \gamma_i v_{i,n} \right) + d_{i,n} \\
					y_{i,n+1} = y_{i,n} + \lambda_n (p_{2,i,n} - y_{i,n})\\
					p_{3,i,n} = J_{\sigma_i B_i^{-1}}\left(v_{i,n} + \sigma_i \left( L_i (2p_{1,n} - x_n) - (2p_{2,i,n} - y_{i,n}) - r_i \right)\right) + b_{i,n}\\
					v_{i,n+1} = v_{i,n} + \lambda_n (p_{3,i,n} - v_{i,n}).
				\end{array} \right.	
		\end{array}
		\right.
		\end{array}
	\end{align}
\end{algorithm}

\begin{theorem}\label{thm2}
In Problem \ref{p1} suppose that
\begin{align}\label{zin2}
	z \in \ran \left( A +  \sum_{i=1}^m L_i^*(B_i\Box D_i)(L_i \cdot -r_i) \right).
\end{align}
and consider the sequences generated by Algorithm \ref{alg2}.
\begin{enumerate}[label={(\roman*)}]
	\setlength{\itemsep}{-2pt}
		\item  \label{th2.01} If
$$\sum_{n=0}^{+\infty}\lambda_n\|a_n\|_{\h} < +\infty, \quad \sum_{n=0}^{+\infty}\lambda_n (\|d_{i,n}\|_{\g_i} + \|b_{i,n}\|_{\g_i} ) < + \infty, i=1,\ldots,m,$$ and
$\sum_{n=0}^{+\infty} \lambda_n (2 - \lambda_n) = +\infty,$
then
\begin{enumerate}[label={(\alph*)}]
	\setlength{\itemsep}{-2pt}
		\item  \label{th2.1} $(x_n,y_{1,n},\ldots,y_{m,n},v_{1,n},\ldots,v_{m,n})_{n\geq 0}$ converges weakly to a point $(\bx,\by_1,...,\by_m,
\bv_1,...,\bv_m)\in\h\times\g_1 \times \ldots \times \g_m\times\g_1 \times \ldots \times \g_m$ such that $(\bx,\bv_1,...,\bv_m)$ is a primal-dual solution to Problem \ref{p1}.
\item  \label{th2.2} $\lambda_n (p_{1,n}-x_{n}) \rightarrow 0 \  (n \rightarrow +\infty)$, $\lambda_n(p_{2,i,n}-y_{i,n}) \rightarrow 0 \ (n \rightarrow +\infty)$ and $\lambda_n(p_{3,i,n}-v_{i,n}) \rightarrow 0 \ (n \rightarrow +\infty)$ for $i=1,...,m$.
		\item \label{th2.3}  whenever $\h$ and $\g_i, i=1,...,m,$ are finite-dimensional Hilbert spaces, $(x_n,v_{1,n},\ldots,v_{m,n})_{n\geq 0}$ converges strongly to a primal-dual solution of Problem \ref{p1}.
		\end{enumerate}
		\item  \label{th2.02} If
$$\sum_{n=0}^{+\infty}\|a_n\|_{\h} < +\infty, \quad \sum_{n=0}^{+\infty} (\|d_{i,n}\|_{\g_i} + \|b_{i,n}\|_{g_i} ) < + \infty, i=1,\ldots,m, \quad \inf_{n\geq 0} \lambda_n > 0$$
$$\mbox{and} \ A, B_i^{-1} \ \mbox{and} \ D_i, i=1,...,m, \ \mbox{are uniformly monotone},$$
then $(p_{1,n}, p_{3,1,n},\ldots,p_{3,m,n})_{n \geq 0}$ converges strongly to the unique primal-dual solution of Problem \ref{p1}.
\end{enumerate}
\end{theorem}

\begin{proof}
We let $\fG =\g_1 \times \ldots \times \g_m$ be the real Hilbert space endowed with the inner product and associated norm defined in \eqref{def1.01} and consider
	$$ \fK = \h \times \fG \times \fG, $$
the real Hilbert space endowed with inner product and associated norm defined for $\f x=(x, \f y, \f v), \f u=(u, \f q, \f p) \in \fK$ as
\begin{align}
	\< \f x, \f u \>_{\fK} = \< x,u \>_{\h} + \< \f y,\f q \>_{\fG} + \< \f v, \f p \>_{\fG} \text{ and }
	\|\f x\|_{\fK}	= \sqrt{\|x\|_{\h}^2 + \|\f y\|_{\fG}^2 + \|\f v\|_{\fG}^2},
\end{align}
respectively. In what follows we set
$$\f y = (y_1,...,y_m), \quad \f v = (v_1,...,v_m), \quad \fby = (\by_1,...,\by_m), \quad \fbv = (\bv_1,...,\bv_m).$$

Consider the set-valued operator
\begin{align*}
	\begin{aligned}
	\f M : \fK &\rightarrow 2^{\fK}, \ (x,\f y, \f v) &\mapsto (-z + Ax, D_1y_1,\ldots,D_m y_m, r_1 + B_1^{-1}v_1, \ldots, r_m + B_m^{-1}v_m),
	\end{aligned}
\end{align*}
which is maximally monotone,  since $A, B_i$ and $D_i, i=1,...,m,$ are maximally monotone (cf. \cite[Proposition 20.22 and Proposition 20.23]{BauCom11}) and the bounded linear operator
\begin{align*}
	\begin{aligned}
	\f S : \fK &\rightarrow \fK, \quad (x, \f y, \f v) &\mapsto \left(\sum_{i=1}^m L_i^* v_i, -v_1, \ldots, -v_m, -L_1 x + y_1 , \ldots, - L_m x + y_m\right),
	\end{aligned}
\end{align*}
which proves to be skew (i.\,e. $\f S^*=-\f S$) and hence maximally monotone (cf. \cite[Example 20.30]{BauCom11}). Since $\dom \f S = \fK$, the sum $\f M + \f S$ is maximally monotone, as well (cf. \cite[Corollary 24.4(i)]{BauCom11}). Further, we have
\begin{align}\label{zer2.1}
		\eqref{zin2} &\Leftrightarrow (\exists\,x \in \h)\ z \in Ax +  \sum_{i=1}^m L_i^*\left( B_i \Box D_i \right)\left(L_i x - r_i\right) \notag \\
		&\Leftrightarrow (\exists\, (x,\f v) \in \h \times \fG) \left\{
				\begin{array}{l}
						z \in A x + \sum_{i=1}^m L_i^* v_i \\
						v_i \in \left( B_i \Box D_i \right)\left(L_i x - r_i\right),\, i=1,\ldots,m
				\end{array}
	\right. \notag \\
		&\Leftrightarrow (\exists\, (x,\f v) \in \h \times \fG) \left\{
				\begin{array}{l}
						z \in A x + \sum_{i=1}^m L_i^* v_i \\
						L_i x - r_i \in B_i^{-1}v_i + D_i^{-1} v_i,\, i=1,\ldots,m
				\end{array}
	\right. \notag \\
		&\Leftrightarrow (\exists\, (x,\f y,\f v) \in \fK) \left\{
				\begin{array}{l}
						0 \in -z + A x + \sum_{i=1}^m L_i^* v_i \\
						0 \in D_i y_i - v_i,  \, i=1,\ldots,m \\
						0 \in r_i + B_i^{-1}v_i - L_i x + y_i,\, i=1,\ldots,m
				\end{array}
	\right. \notag \\	
	 &\Leftrightarrow (\exists\, (x,\f y,\f v) \in \fK)\ (0,\ldots,0) \in \left(\f M + \f S\right)(x,\f y,\f v) \notag \\
	 &\Leftrightarrow \zer\left(\f M + \f S\right) \neq \varnothing.
\end{align}
From the above calculations it follows that
{\allowdisplaybreaks
\begin{align}
	(x,\fy, \fv) \in \zer\left(\f M + \f S \right)
	& \Rightarrow \left\{
				\begin{array}{l}
						z - \sum_{i=1}^m L_i^* v_i \in A x \\
						v_i \in \left( B_i \Box D_i \right)\left(L_i x - r_i\right), i=1,...,m
				\end{array}\right. \notag\\
	&\Leftrightarrow (x,v_1,...,v_m) \text{ is a primal-dual solution to Problem }\ref{p1}. \label{optcond2-p}
\end{align}}

Finally, we introduce the bounded linear operator
\begin{align*}
	\f V : \fK &\rightarrow \fK \\
	(x,\f y, \f v) &\mapsto \left(\frac{x}{\tau} - \sum_{i=1}^m L_i^* v_i,\frac{y_1}{\gamma_1} +v_1,\ldots,\frac{y_m}{\gamma_m} +v_m, \frac{v_1}{\sigma_1} - L_1x +y_1, \ldots, \frac{v_m}{\sigma_m} - L_m x + y_m \right),
\end{align*}
which is self-adjoint, i.\,e. $\f V^* = \f V$. Furthermore, the operator $\f V$ is $\rho$-strongly positive for
$$\rho =\left(1- 2\sqrt{\tau \sum_{i=1}^m \sigma_i \|L_i\|^2}\right)\min\left\{\frac{1}{\tau},\frac{1}{\gamma_1},\ldots,\frac{1}{\gamma_m},\frac{1}{\sigma_1},\ldots,\frac{1}{\sigma_m}\right\},$$
which is a positive real number due to the assumption
\begin{align}
	\tau \sum_{i=1}^m \sigma_i \|L_i\|^2 < \frac{1}{4}
\end{align}
made in Algorithm \ref{alg2}. Indeed, for $\gamma_i=\sigma_i^{-1}\tau\sum_{i=1}^m\sigma_i\|L_i\|^2$ it yields for each $i=1,\ldots,m$,
\begin{align*}
	2\<L_i x-y_i,v_i\>_{\g_i} &\leq 2\|L_i\|\|x\|_{\h} \|v_i\|_{\g_i} + 2\|y_i\|_{\g_i} \|v_i\|_{\g_i} \\
	&\leq \frac{\sigma_i \|L_i\|^2\|x\|_{\h}^2}{\sqrt{\tau \sum_{i=1}^m \sigma_i \|L_i\|^2}}  + \sqrt{\tau \sum_{i=1}^m \sigma_i \|L_i\|^2} \frac{\|y_i\|_{\g_i}^2}{\gamma_i} + 2 \sqrt{\tau \sum_{i=1}^m \sigma_i \|L_i\|^2} \frac{\|v_i\|_{\g_i}^2}{\sigma_i}
\end{align*}
and, consequently, for each $\f x=(x,\f y, \f v)\in\fK$, it follows that
\begin{align}
 \< \f x, \f V \f x \>_{\fK} &
= \frac{\|x\|_{\h}^2}{\tau} + \sum_{i=1}^m \left [\frac{\|y_i\|_{\g_i}^2}{\gamma_i} + \frac{\|v_i\|_{\g_i}^2}{\sigma_i}\right] - 2\sum_{i=1}^m \<L_i x-y_i,v_i \>_{\g_i} \notag \\	
&\geq \left(1- 2\sqrt{\tau \sum_{i=1}^m \sigma_i \|L_i\|^2}\right)\min\left\{\frac{1}{\tau},\frac{1}{\gamma_1},\ldots,\frac{1}{\gamma_m},\frac{1}{\sigma_1},\ldots,\frac{1}{\sigma_m}\right\} \|\f x\|_{\fK}^2 \notag \\
&=\rho \|\f x\|_{\fK}^2.
\end{align}

The algorithmic scheme \eqref{A2} is equivalent to
\begin{align}
	\label{A2.1}
		\left(\forall n\geq 0\right) \begin{array}{l} \left\lfloor \begin{array}{l}
		\frac{x_n - p_{1,n}}{\tau} - \sum_{i=1}^m L_i^* v_{i,n} \in -z + A(p_{1,n}-a_n) -\frac{a_n}{\tau} \\
		x_{n+1} = x_n + \lambda_n ( p_{1,n} - x_n ) \\
		\text{For }i=1,\ldots,m  \\
		\left\lfloor \begin{array}{l}
					\frac{y_{i,n}-p_{2,i,n}}{\gamma_i} + v_{i,n} \in D_i(p_{2,i,n}-d_{i,n}) - \frac{d_{i,n}}{\gamma_i} \\
					y_{i,n+1} = y_{i,n} + \lambda_n (p_{2,i,n} - y_{i,n})\\
					\frac{v_{i,n}-p_{3,i,n}}{\sigma_i} -L_i(x_n - p_{1,n}) + y_{i,n} - p_{2,i,n}  \\ \qquad \qquad \qquad \in r_i + B_i^{-1}(p_{3,i,n}-b_{i,n}) -L_ip_{1,n} +p_{2,i,n} -\frac{b_{i,n}}{\sigma_i} \\
					v_{i,n+1} = v_{i,n} + \lambda_n (p_{3,i,n} - v_{i,n}).
				\end{array} \right.	
		\end{array}
		\right.
		\end{array}
\end{align}
We introduce for every $n \geq 0$ the following notations:
\begin{align}
 \left\{
	\begin{array}{l}
	\fx_n = (x_n,y_{1,n},\ldots,y_{m,n},v_{1,n},\ldots,v_{m,n}) \\
	\fp_n = (p_{1,n},p_{2,1,n},\ldots,p_{2,m,n},p_{3,1,n},\ldots,p_{3,m,n}) \\
	\f a_n = (a_n,d_{1,n},\ldots,d_{m,n},b_{1,n},\ldots,b_{m,n}) \\
	\f a_n^{\tau} = (\frac{a_n}{\tau},\frac{d_{1,n}}{\gamma_1},\ldots,\frac{d_{m,n}}{\gamma_m},\frac{b_{1,n}}{\sigma_1},\ldots,\frac{b_{m,n}}{\sigma_m}).
	\end{array}\right.
\end{align}
Hence, the scheme \eqref{A2.1} can equivalently be written in the form
\begin{align}
	\label{A2.2}
		\left(\forall n\geq 0\right)  \left\lfloor \begin{array}{l}
		\f V(\f x_n-\f p_n) \in \left( \f S + \f M \right) \left(\fp_n - \f a_n\right) + \f S \f a_n - \f a_n^{\tau} \\
		\f x_{n+1} = \f x_n + \lambda_n ( \f p_n - \f x_n).
		\end{array}
		\right.
\end{align}

Considering again the Hilbert space $\fK_{\f V}$ with inner product and norm respectively defined as in \eqref{HSKV}, since $\f V$ is self-adjoint and $\rho$-strongly positive,  weak and strong convergence in $\fK_{\f V}$ are equivalent with weak and strong convergence in $\fK$, respectively. Moreover,
$\f A=\f V^{-1}\left( \f S + \f M \right)$ is maximally monotone on $\fK_{\f V}$. Thus, by denoting $\f e_n = \f V^{-1}\left(\left(\f S + \f V\right)\f a_n - \f a_n^{\tau}\right)$ for every $n \geq 0$ the iterative scheme \eqref{A2.2} becomes
 \begin{align}
	\label{A2.3}
		\left(\forall n\geq 0\right)  \left\lfloor \begin{array}{l}
	\fp_n = J_{\f A}\left(\fx_n - \f e_n\right) + \f a_n\\
	\f x_{n+1} = \f x_n + \lambda_n ( \f p_n - \f x_n).
	\end{array}
	\right.
\end{align}

Furthermore, introducing the maximal monotone operator $\f B : \fK \rightarrow 2^{\fK}$, $\fx \mapsto \{0\},$ and defining for every $n \geq 0$
$$\f \alpha_n = J_{\f A}\left(\fx_n - \f e_n\right) - J_{\f A}\left(\fx_n\right) +  \f a_n,$$
the iterative scheme \eqref{A2.3} becomes (notice that $J_{\f B}=\id$)
\begin{align}
	\label{A2.4}
	\left(\forall n\geq 0\right)  \left\lfloor \begin{array}{l}
	\fy_n = J_{\f B}(\fx_n) \\
	\fp_n = J_{\f A}\left(2\fy_n - \fx_n\right) + \f \alpha_n \\
	\fx_{n+1} = \fx_n + \lambda_n(\fp_n-\fy_n)	
	\end{array}
	\right.,
\end{align}
thus, it has the structure of the error-tolerant Douglas-Rachford algorithm from \cite{Com09}. Obviously, $\zer(\f A + \f B) = \zer(\f M + \f S)$.

\ref{th2.01} The assumptions made on the error sequences yield
\begin{align*}
	\sum_{n=0}^{+\infty} \lambda_n \|\f a_n\|_{\fK} < +\infty \ \mbox{and} \  \sum_{n=0}^{+\infty} \lambda_n\|\f e_n\|_{\fK} < +\infty,
\end{align*}
thus, by the nonexpansiveness of the resolvent of $\f A$,
\begin{align*}
	\sum_{n=0}^{+\infty} \lambda_n \|\f \alpha_n\|_{\fK}  < +\infty
\end{align*}
and, consequently, by the linearity and boundedness of $\f V$,
\begin{align*}
	\sum_{n=0}^{+\infty} \lambda_n \|\f \alpha_n\|_{\fK_{\f V}}  < +\infty.
\end{align*}

\ref{th2.01}\ref{th2.1} Follows directly from \cite[Theorem 2.1(i)(a)]{Com09} by using that $J_{\f B}=\id$ and relation \eqref{optcond2-p}.

\ref{th2.01}\ref{th2.2} Follows in analogy to the proof of Theorem \ref{thm1}\ref{th1.01}\ref{th1.2}.

\ref{th2.01}\ref{th2.3} Follows from Theorem \ref{thm2}\ref{th2.01}\ref{th2.1}.

\ref{th2.02} The iterative scheme \eqref{A2.3} can be also formulated as
\begin{align}
	\label{A2.5}
	\left(\forall n\geq 0\right)  \left\lfloor \begin{array}{l}
	\fp_n = J_{\f A}(\fx_n) + \f \alpha_n \\
	\fy_n = J_{\f B}\left(2\fp_n - \fx_n\right)\\
	\fx_{n+1} = \fx_n + \lambda_n(\fy_n-\fp_n)	
	\end{array}
	\right.,
\end{align}
with the error sequence fulfilling
\begin{align*}
	\sum_{n=0}^{+\infty} \|\f \alpha_n\|_{\fK_{\f V}}  < +\infty.
\end{align*}
The statement follows from \cite[Theorem 2.1(ii)(b)]{Com09} by taking into consideration the uniform monotonicity of $\f A$ and relation \eqref{optcond2-p}.
\end{proof}

\begin{remark}\label{rm2.1}
	When for every $i=1,...,m$ one takes $D_i(0) = \g_i$ and $D_i(v) = \varnothing \ \forall v \in \g_i \setminus \{0\}$, and $(d_{i,n})_{n \geq 0}$  as a sequence of zeros, one can show that the assertions made in Theorem \ref{thm2} hold true for step length parameters satisfying
	\begin{align*}
		\tau \sum_{i=1}^m \sigma_i \|L_i\|^2 < 1,
	\end{align*}
when choosing $(y_{1,0},\ldots,y_{m,0})=(0,\ldots,0)$ in Algorithm \ref{alg2}, since the sequences $(y_{1,n},\ldots,y_{m,n})_{n \geq 0}$
and $(v_{1,n},\ldots,v_{m,n})_{n \geq 0}$ vanish in this particular situation.
\end{remark}
\begin{remark}\label{rm2.2}
 In the following we emphasize the relations between Algorithm \ref{alg2} and other existent primal-dual iterative schemes.	
	\begin{enumerate}[label={(\roman*)}]
		\setlength{\itemsep}{-2pt}
		\item \label{rm2.2.1}  When for every $i=1,...,m$ one takes $D_i(0) = \g_i$ and $D_i(v) = \varnothing \ \forall v \in \g_i \setminus \{0\}$, and $(d_{i,n})_{n \geq 0}$ as a sequence of zeros, Algorithm \ref{alg2} with $(y_{1,0},\ldots,y_{m,0})=(0,\ldots,0)$ as initial choice provides an iterative scheme which is identical to the one in \cite[Eq. (3.3)]{Vu11}, but differs from the one in  \cite[Theorem 3.1]{ComPes12} (see, also, \cite[Theorem 3.1]{BriCom11} for the case $m=1$) when  the latter are applied to Problem \ref{p1}.
       \item \label{rm2.2.4} When solving the particular case of a primal-dual pair of convex optimization problems discussed in Example \ref{ex1} and when considering as initial choice  $y_{1,0}=0$, Algorithm \ref{alg2} gives rise to an iterative scheme which is equivalent to \cite[Algorithm 3.1]{Con12}. In addition, under the assumption of exact implementations, the method in Algorithm \ref{alg2} equals the one in \cite[Algorithm 1]{ChaPoc11}, our choice of $(\lambda_n)_{n \geq 0}$ to be variable in the interval $(0,2)$, however, relaxes the assumption in \cite{ChaPoc11} that $(\lambda_n)_{n \geq 0}$ is a constant sequence in  $(0,1]$.
\end{enumerate}
\end{remark}

\section{Application to convex minimization problems}\label{sectionMin}
In this section we particularize the two iterative schemes introduced and investigated in this paper in the context of solving a primal-dual pair of convex optimization problems. To this end we consider the following problem.
\begin{problem}\label{p1.2}
Let $\h$ and $\g_i$, $i=1,...,m$, be given real Hilbert spaces, $f \in \Gamma(\h)$ and $z \in \h$,  $g_i, l_i \in \Gamma(\g_i)$, $r_i \in \g_i$, $i=1,...,m$, and $L_i:\h \rightarrow \g_i$, $i=1,...,m,$ nonzero bounded linear operators. Consider the convex optimization problem
\begin{align}
	\label{opt-mp}
	(P) \quad \inf_{x \in \h}{\left\{f(x)+\sum_{i=1}^m (g_i \Box l_i)(L_ix-r_i) -\<x,z\>\right\}}
\end{align}
and its conjugate dual problem
\begin{align}
	\label{opt-md}
	(D) \quad \sup_{(v_1,\ldots,v_m) \in \g_1\times\ldots\times\g_m}{\left\{-f^*\left( z - \sum_{i=1}^m L_i^*v_i\right) - \sum_{i=1}^m \left( g_i^*(v_i) + l_i^*(v_i) + \< v_i,r_i \> \right) \right\} }.
\end{align}
\end{problem}
Considering the maximal monotone operators
\begin{align*}
	 A=\partial f, \ B_i = \partial g_i \text{ and }	D_i = \partial l_i,\ i=1,...,m,
\end{align*}
the monotone inclusion problem \eqref{opt-p} reads
\begin{align}
	\label{opt-po}
	\text{find }\bx \in \h \text{ such that } z \in \partial f(\bx) + \sum_{i=1}^m L_i^* (\partial g_i \Box \partial l_i)(L_i \bx-r_i),
\end{align}
while the dual inclusion problem \eqref{opt-d} reads
\begin{align}
	\label{opt-do}
	\text{find }\bv_1 \in \g_1,\ldots,\bv_m \in \g_m \text{ such that }(\exists x\in\h)\left\{
	\begin{array}{l}
		z - \sum_{i=1}^m L_i^*\bv_i \in \partial f(x) \\
		\!\!\bv_i \in \!\! (\partial g_i \Box \partial l_i)(L_ix-r_i), \,i=1,\ldots,m.
	\end{array}
\right.
\end{align}
If $(\bx, \bv_1,\ldots,\bv_m) \in \h \times \g_1 \ldots \times \g_m$ is a primal-dual solution to \eqref{opt-po}-\eqref{opt-do}, namely, \begin{equation}\label{operator-proof-conditions-full-0}
z - \sum_{i=1}^m L_i^*\bv_i \in \partial f(\bx) \ \mbox{and} \  \bv_i \in (\partial g_i \Box \partial l_i)(L_i\bx-r_i), \,i=1,\ldots,m,
\end{equation}
then $\bx$ is an optimal solution to $(P)$, $(\bv_1,\ldots,\bv_m)$ is an optimal solution to $(D)$ and the optimal objective values of the two problems, which we denote by $v(P)$ and $v(D)$, respectively, coincide (thus, strong duality holds).

Combining this statement with Algorithm \ref{alg1} and Theorem \ref{thm1} give rise to the following iterative scheme and corresponding convergence results for
the primal-dual pair of optimization problems $(P)-(D)$. We also use that the subdifferential of a uniformly convex function is uniformly monotone (cf. \cite[Example 22.3(iii)]{BauCom11}).

\begin{algorithm}\label{alg3} \text{ }\newline
Let $x_0 \in \h$, $(v_{1,0}, \ldots, v_{m,0}) \in \g_1 \times \ldots \times \g_m$ and $\tau$ and $\sigma_i$, $i=1,...,m,$ be strictly positive real numbers such that
$$\tau \sum_{i=1}^m \sigma_i \|L_i\|^2 < 4. $$
Furthermore, let $(\lambda_n)_{n\geq 0}$ be a sequence in $(0,2)$, $(a_n)_{n\geq 0}$ a sequence in $\h$, $(b_{i,n})_{n\geq 0}$ and $(d_{i,n})_{n\geq 0}$ sequences in $\g_i$ for all $i=1,\ldots,m$ and set
	\begin{align}\label{A3}
	  \left(\forall n\geq 0\right) \begin{array}{l}  \left\lfloor \begin{array}{l}
		p_{1,n} = \Prox_{\tau f}\left( x_n - \frac{\tau}{2} \sum_{i=1}^m L_i^* v_{i,n} + \tau z \right) + a_n \\
		w_{1,n} = 2p_{1,n} - x_n \\
		\text{For }i=1,\ldots,m  \\
				\left\lfloor \begin{array}{l}
					p_{2,i,n} = \Prox_{\sigma_i g_i^*}\left(v_{i,n} +\frac{\sigma_i}{2} L_i w_{1,n} - \sigma_i r_i \right) + b_{i,n} \\
					w_{2,i,n} = 2 p_{2,i,n} - v_{i,n} \\
				\end{array} \right.\\
		z_{1,n} = w_{1,n} - \frac{\tau}{2} \sum_{i=1}^m L_i^* w_{2,i,n} \\
		x_{n+1} = x_n + \lambda_n ( z_{1,n} - p_{1,n} ) \\
		\text{For }i=1,\ldots,m  \\
				\left\lfloor \begin{array}{l}
					z_{2,i,n} = \Prox_{\sigma_i l_i^*}\left(w_{2,i,n} + \frac{\sigma_i}{2}L_i (2 z_{1,n} - w_{1,n}) \right) + d_{i,n} \\
					v_{i,n+1} = v_{i,n} + \lambda_n (z_{2,i,n} - p_{2,i,n}). \\
				\end{array} \right. \\ \vspace{-4mm}
		\end{array}
		\right.
		\end{array}
	\end{align}
\end{algorithm}

\begin{theorem}\label{thm3}
For Problem \ref{p1.2} suppose that
\begin{align}\label{zin1.2}
	z \in \ran \left( \partial f +  \sum_{i=1}^m L_i^*(\partial g_i \Box \partial l_i)(L_i \cdot -r_i) \right),
\end{align}
and consider the sequences generated by Algorithm \ref{alg3}.
\begin{enumerate}[label={(\roman*)}]
	\setlength{\itemsep}{-2pt}
		\item  \label{th3.01} If
$$\sum_{n=0}^{+\infty}\lambda_n\|a_n\|_{\h} < +\infty, \quad \sum_{n=0}^{+\infty}\lambda_n (\|d_{i,n}\|_{\g_i} + \|b_{i,n}\|_{\g_i} ) < + \infty, i=1,\ldots,m,$$ and
$\sum_{n=0}^{+\infty} \lambda_n (2 - \lambda_n) = +\infty,$
then
		\begin{enumerate}[label={(\roman*)}]
	\setlength{\itemsep}{-2pt}
		\item  \label{th3.1} $(x_n,v_{1,n},\ldots,v_{m,n})_{n\geq 0}$ converges weakly to a point $(\bx,\bv_1,\ldots,\bv_m)\in\h\times\g_1\times\ldots\times\g_m$ such that, when setting
		\begin{align*}
						\bp_1 &= \Prox\nolimits_{\tau f}\left( \bx - \frac{\tau}{2} \sum_{i=1}^m L_i^* \bv_i + \tau z \right), \\
						\text{and }\bp_{2,i} & =\Prox\nolimits_{\sigma_i g_i^*}\left(\bv_{i} +\frac{\sigma_i}{2} L_i (2\bp_1-\bx) -\sigma_i r_i \right) \ i =1,...,m,
		\end{align*}
$\bp_1$ is an optimal solution to $(P)$, $(\bp_{2,1},\ldots,\bp_{2,m})$ is an optimal solution to $(D)$ and $v(P) = v(D)$.
		\item  \label{th3.2}$\lambda_n (z_{1,n}-p_{1,n}) \rightarrow 0 \  (n \rightarrow +\infty)$ and  $\lambda_n(z_{2,i,n}-p_{2,i,n}) \rightarrow  0 \  (n \rightarrow +\infty)$ for $i=1,...,m$.
		\item  \label{th3.3} whenever $\h$ and $\g_i, i=1,...,m,$ are finite-dimensional Hilbert spaces, $a_n \rightarrow 0 \  (n \rightarrow +\infty)$ and $b_{i,n} \rightarrow 0 \  (n \rightarrow +\infty)$ for $i=1,...,m$, then $(p_{1,n})_{n \geq 0}$ converges strongly to an optimal solution to $(P)$ and  $(p_{2,1,n},\ldots,p_{2,m,n})_{n \geq 0}$ converges strongly to an optimal solution to $(D)$.
		\end{enumerate}
		\item  \label{th3.02}  If
$$\sum_{n=0}^{+\infty}\|a_n\|_{\h} < +\infty, \quad \sum_{n=0}^{+\infty} (\|d_{i,n}\|_{\g_i} + \|b_{i,n}\|_{g_i} ) < + \infty, i=1,\ldots,m, \quad \inf_{n\geq 0} \lambda_n > 0$$
$$\mbox{and} \ f \ \mbox{and} \ g_i^*, i=1,...,m, \ \mbox{are uniformly convex},$$
then $(p_{1,n})_{n \geq 0}$ converges strongly to an optimal solution to $(P)$, $(p_{2,1,n},\ldots,p_{2,m,n})_{n \geq 0}$ converges strongly to an optimal solution to $(D)$ and $v(P) = v(D)$.
\end{enumerate}
\end{theorem}

Algorithm \ref{alg2} and Theorem \ref{thm2} give rise to the following iterative scheme and corresponding convergence results for
the primal-dual pair of optimization problems $(P)-(D)$.

\begin{algorithm}\label{alg4} \text{ }\newline
	Let $x_0 \in \h$, $(y_{1,0}, \ldots, y_{m,0}) \in \g_1 \times \ldots \times \g_m$, $(v_{1,0}, \ldots, v_{m,0}) \in \g_1 \times \ldots \times \g_m$, and $\tau$ and $\sigma_i$, $i=1,...,m,$ be strictly positive real numbers such that
$$	\tau \sum_{i=1}^m \sigma_i \|L_i\|^2 < \frac{1}{4}.$$
Furthermore, set $\gamma_i=\sigma_i^{-1}\tau\sum_{i=1}^m \sigma_i\|L_i\|^2$, $i=1,...,m$,  $(\lambda_n)_{n\geq 0}$ be a sequence in $(0,2)$, $(a_n)_{n\geq 0}$ a sequence in $\h$, $(b_{i,n})_{n\geq 0}$ and $(d_{i,n})_{n\geq 0}$ sequences in $\g_i$ for all $i=1,\ldots,m$ and set
	\begin{align}\label{A4}
	  \left(\forall n\geq 0\right) \begin{array}{l} \left\lfloor \begin{array}{l}
		p_{1,n} = \Prox_{\tau f}\left( x_n - \tau\left( \sum_{i=1}^m L_i^* v_{i,n} - z \right)\right) + a_n \\
		x_{n+1} = x_n + \lambda_n ( p_{1,n} - x_n ) \\
		\text{For }i=1,\ldots,m  \\
		\left\lfloor \begin{array}{l}
					p_{2,i,n} = \Prox_{\gamma_i l_i}\left(y_{i,n} + \gamma_i v_{i,n} \right) + d_{i,n} \\
					y_{i,n+1} = y_{i,n} + \lambda_n (p_{2,i,n} - y_{i,n})\\
					p_{3,i,n} = \Prox_{\sigma_i g_i^*}\left(v_{i,n} + \sigma_i \left( L_i (2p_{1,n} - x_n) - (2p_{2,i,n} - y_{i,n}) - r_i \right)\right) + b_{i,n} \\
					v_{i,n+1} = v_{i,n} + \lambda_n (p_{3,i,n} - v_{i,n}).
				\end{array} \right.	
		\end{array}
		\right.
		\end{array}
	\end{align}
\end{algorithm}

\begin{theorem}\label{thm4}
In Problem \ref{p1.2} suppose that
\begin{align}\label{zin2.2}
	z \in \ran \left( \partial f +  \sum_{i=1}^m L_i^*(\partial g_i \Box \partial l_i)(L_i \cdot -r_i) \right),
\end{align}
and consider the sequences generated by Algorithm \ref{alg4}.
\begin{enumerate}[label={(\roman*)}]
	\setlength{\itemsep}{-2pt}
		\item  \label{th4.01} If
$$\sum_{n=0}^{+\infty}\lambda_n\|a_n\|_{\h} < +\infty, \quad \sum_{n=0}^{+\infty}\lambda_n (\|d_{i,n}\|_{\g_i} + \|b_{i,n}\|_{\g_i} ) < + \infty, i=1,\ldots,m,$$ and
$\sum_{n=0}^{+\infty} \lambda_n (2 - \lambda_n) = +\infty,$
then
		\begin{enumerate}[label={(\alph*)}]
	\setlength{\itemsep}{-2pt}
		\item  \label{th4.1} $(x_n,y_{1,n},\ldots,y_{m,n},v_{1,n},\ldots,v_{m,n})_{n\geq 0}$ converges weakly to a point $(\bx,\by_1,...,\by_m,
\bv_1,...,\bv_m)\in\h\times\g_1 \times \ldots \times \g_m\times\g_1 \times \ldots \times \g_m$ such that such that $\bx$ is an optimal solution to $(P)$, $(\bv_1,...,\bv_m)$ is an optimal solution to $(D)$ and $v(P) = v(D)$.
		\item  \label{th4.2} $\lambda_n (p_{1,n}-x_{n}) \rightarrow 0 \  (n \rightarrow +\infty)$, $\lambda_n(p_{2,i,n}-y_{i,n}) \rightarrow 0 \ (n \rightarrow +\infty)$ and $\lambda_n(p_{3,i,n}-v_{i,n}) \rightarrow 0 \ (n \rightarrow +\infty)$ for $i=1,...,m$.
		\item  \label{th4.3}  whenever $\h$ and $\g_i$, $i=1,...,m,$ are finite-dimensional Hilbert spaces, $(x_n)_{n\geq 0}$ converges strongly to an optimal solution to $(P)$ and $(v_{1,n},\ldots,v_{m,n})_{n\geq 0}$ converges strongly to an optimal solution to $(D)$.
		\end{enumerate}
		\item  \label{th4.02} If
$$\sum_{n=0}^{+\infty}\|a_n\|_{\h} < +\infty, \quad \sum_{n=0}^{+\infty} (\|d_{i,n}\|_{\g_i} + \|b_{i,n}\|_{g_i} ) < + \infty, i=1,\ldots,m, \quad \inf_{n\geq 0} \lambda_n > 0$$
$$\mbox{and} \ f, l_i \ \mbox{and} \ g_i^*, i=1,...,m, \ \mbox{are uniformly convex},$$
then,  $(p_{1,n})_{n\geq 0}$ converges strongly to the unique optimal solution to $(P)$, $(p_{3,1,n},\ldots,p_{3,m,n})_{n \geq 0}$ converges strongly to the unique optimal solution of $(D)$ and $v(P) = v(D)$.
	\end{enumerate}
\end{theorem}

\begin{remark}
\label{remark-alg4}
According to Remark \ref{rm2.1}, when $l_i:\g_i \rightarrow \oR$, $l_i=\delta_{\{0\}}$, and $(d_{i,n})_{n \geq 0}$ is chosen as a sequence of zeros for every $i =1,...,m$, the assertions made in Theorem \ref{thm4} hold true for step length parameters satisfying
	\begin{align*}
		\tau \sum_{i=1}^m \sigma_i \|L_i\|^2 < 1
	\end{align*}
when taking in Algorithm \ref{alg4} as initial choice $(y_{1,0},\ldots,y_{m,0})=(0,\ldots,0)$. In this case the sequences $(y_{1,n},\ldots,y_{m,n})_{n\geq 0}$ and $(p_{2,1,n},\ldots,p_{2,m,n})_{n\geq 0}$ vanish and \eqref{A4} reduces to
	\begin{align}\label{A4.2}
	  \left(\forall n\geq 0\right)
		\begin{array}{l} \left\lfloor
			\begin{array}{l}
				p_{1,n} = \Prox_{\tau f}\left( x_n - \tau\left( \sum_{i=1}^m L_i^* v_{i,n} - z \right)\right) + a_n \\
				x_{n+1} = x_n + \lambda_n ( p_{1,n} - x_n )\\
				\text{For }i=1,\ldots,m  \\
				\left\lfloor
				\begin{array}{l}
					p_{3,i,n} = \Prox_{\sigma_i g_i^*}\left(v_{i,n} + \sigma_i \left( L_i (2p_{1,n} - x_n) - r_i \right)\right) + b_{i,n} \\
					v_{i,n+1} = v_{i,n} + \lambda_n (p_{3,i,n} - v_{i,n}).
				\end{array} \right. \\ \vspace{-4mm}
			\end{array}	\right.
		\end{array}
	\end{align}
\end{remark}

\section{Numerical experiments}\label{sectionEx}

In this section we emphasize the performances of the algorithms introduced in this article in the context of two numerical experiments on location and image deblurring problems.

\subsection{The generalized Heron problem}\label{subsectionHeron}

We start by considering the \textit{generalized Heron problem} which has been recently investigated in \cite{MorNamSal12a,MorNamSal12b} and where for its solving
subgradient-type methods have been used.

While the \textit{classical Heron problem} concerns the finding of a point $\bu$ on a given straight line in the plane such that the sum of distances from $\bu$ to given points $u^1,\,u^2$ is minimal, the problem that we address here aims to find a point in a closed convex set $\Omega \subseteq \R^n$ which minimizes the sum of the distances to given convex closed sets $\Omega_i \subseteq \R^n$, $i=1,\ldots,m$.

The distance from a point $x \in \R^n$ to a nonempty set $\Omega \subseteq  \R^n$ is given by
\begin{align*}
	d(x;\Omega)=(\|\cdot\| \Box \delta_{\Omega})(x) = \inf_{z\in\Omega}\|x-z\|.
\end{align*}
Thus the \textit{generalized Heron problem} reads
\begin{align}
	\label{ex-p1}
	\inf_{x\in\Omega} \sum_{i=1}^m d(x;\Omega_i) ,
\end{align}
where the sets $\Omega \subseteq \R^n$ and $\Omega_i \subseteq \R^n$, $i=1,\ldots,m$, are nonempty, closed and convex. We observe that \eqref{ex-p1} perfectly fits into the framework considered in Problem \ref{p1.2} when setting
\begin{align}
	\label{ex-f1}
	f=\delta_{\Omega}, \text{ and } g_i=\|\cdot\|,\ l_i=\delta_{\Omega_i} \text{ for all }i=1,\ldots,m.
\end{align}
However, note that \eqref{ex-p1} cannot be solved via the primal-dual methods in \cite{ComPes12} and \cite{Vu11} since they require the presence of at least one strongly convex function (cf. Baillon-Haddad Theorem, \cite[Corollary 18.16]{BauCom11}) in each of the infimal convolutions $\|\cdot\| \Box \delta_{\Omega_i}$, $i=1,\ldots,m$, fact which is obviously not the case. Notice that
$$g_i^*:\R^n \rightarrow \oR,\ g_i^*(p)=\sup_{x\in\R^n}\left\{\<p,x\> - \|x\|\right\} = \delta_{B(0,1)}(p), i=1,...,m,$$ thus
the proximal points of $f$, $g_i^*$ and $l_i^*, i=1,...,m,$ can be calculated via projections, in case of the latter via Moreau's decomposition formula.

We test our algorithms on some examples taken from \cite{MorNamSal12a,MorNamSal12b} and denote in the final evaluations for every $k \geq 0$ by $V_k$ the value of the objective function of \eqref{ex-p1} at each iterate.

\begin{example}[Example 5.5 in \cite{MorNamSal12b}]\label{num-ex1}
Consider problem \eqref{ex-p1} with the constraint set $\Omega$ being the closed ball centered at $(5,0)$ having radius $2$ and the sets $\Omega_i$, $i=1,\ldots,8$, being pairwise disjoint squares in \textit{right position} in $\R^2$ (i.\,e. the edges are parallel to the x- and y-axes, respectively), with centers $(-2,4)$, $(-1,-8)$, $(0,0)$, $(0,6)$, $(5,-6)$, $(8,-8)$, $(8,9)$ and $(9,-5)$ and side length $1$, respectively (see Figure \ref{fig:ex1}).

When solving this problem with Algorithm \ref{alg3} and Algorithm \ref{alg4} and the choices made in \eqref{ex-f1}, the following formulae for the proximal points involved in their formulations are necessary:
\begin{align*}
	\Prox\nolimits_{\tau f}\left(x\right) &= (5,0) + \argmin_{y \in B(0,2)}\frac{1}{2}\left\|y-\left(x-(5,0)\right)\right\|^2 = (5,0) + \proj_{B(0,2)}\left(x-(5,0)\right) \\
	\Prox\nolimits_{\sigma_i g_i^*}(p) &= \argmin_{z \in B(0,1)} \frac{1}{2}\|z-p\|^2 = \proj_{B(0,1)}\left(p\right) \\
	\Prox\nolimits_{\sigma_i l_i^*}(p) &\overset{\mathclap{\eqref{res-indentity}}}{=} p - \sigma_i \Prox\nolimits_{\sigma_i^{-1} l_i}\left(\frac{p}{\sigma_i}\right) = p - \sigma_i \argmin_{z\in\Omega_i} \frac{1}{2}\left\|z-\frac{p}{\sigma_i}\right\|^2= p - \sigma_i\proj_{\Omega_i}\left(\frac{p}{\sigma_i}\right).
\end{align*}

Hence, when choosing $x_0\in\R^2$ and $(v_{1,0},\ldots,v_{8,0})\in \R^2 \times \ldots \times \R^2$ as starting values, letting $(\lambda_n)_{n\geq 0} \subseteq (0,2)$ and $\tau,\sigma_i\in \R_{++}$, $i=1,\ldots,8$, such that $\tau \sum_{i=1}^8 \sigma_i < 4$, the iterative scheme Algorithm \ref{alg3} becomes
\begin{align*}
	  \left(\forall n\geq 0\right) \begin{array}{l}  \left\lfloor \begin{array}{l}
		p_{1,n} = (5,0) + \proj_{B(0,2)}\left(x_n - \frac{\tau}{2} \sum_{i=1}^8 v_{i,n}-(5,0)\right)\\
		w_{1,n} = 2p_{1,n} - x_n \\
		\text{For }i=1,\ldots,8  \\
				\left\lfloor \begin{array}{l}
					p_{2,i,n} = \proj_{B(0,1)}\left(v_{i,n} +\frac{\sigma_i}{2} w_{1,n} \right)  \\
					w_{2,i,n} = 2 p_{2,i,n} - v_{i,n} \\
				\end{array} \right.\\
		z_{1,n} = w_{1,n} - \frac{\tau}{2} \sum_{i=1}^8 w_{2,i,n} \\
		x_{n+1} = x_n + \lambda_n ( z_{1,n} - p_{1,n} ) \\
		\text{For }i=1,\ldots,8  \\
				\left\lfloor \begin{array}{l}
					z_{2,i,n} = w_{2,i,n} + \frac{\sigma_i}{2} (2 z_{1,n} - w_{1,n}) - \sigma_i\proj_{\Omega_i}\left(\frac{w_{2,i,n}}{\sigma_i}+ \frac{1}{2} (2 z_{1,n} - w_{1,n}) \right)   \\
					v_{i,n+1} = v_{i,n} + \lambda_n (z_{2,i,n} - p_{2,i,n}). \\
				\end{array} \right. \\ \vspace{-4mm}
		\end{array}
		\right.
		\end{array}
	\end{align*}
For the same initial choices, but when  $\tau \sum_{i=1}^8 \sigma_i < \frac{1}{4}$, the iterative scheme in Algorithm \ref{alg4} reads
	\begin{align*}
	  \left(\forall n\geq 0\right) \begin{array}{l} \left\lfloor \begin{array}{l}
		p_{1,n} = (5,0) + \proj_{B(0,2)}\left(x_n - \tau\sum_{i=1}^8 v_{i,n}-(5,0)\right)   \\
		x_{n+1} = x_n + \lambda_n ( p_{1,n} - x_n ) \\
		\text{For }i=1,\ldots,8  \\
		\left\lfloor \begin{array}{l}
					p_{2,i,n} = \proj_{\Omega_i}\left(y_{i,n} + \gamma_i v_{i,n}\right)  \\
					y_{i,n+1} = y_{i,n} + \lambda_n (p_{2,i,n} - y_{i,n})\\
					p_{3,i,n} = \proj_{B(0,1)}\left(v_{i,n} + \sigma_i \left( (2p_{1,n} - x_n) - (2p_{2,i,n} - y_{i,n}) \right) \right)  \\
					v_{i,n+1} = v_{i,n} + \lambda_n (p_{3,i,n} - v_{i,n}).
				\end{array} \right.	
		\end{array}
		\right.
		\end{array}
	\end{align*}	

\begin{figure}[htb]
	\begin{minipage}[t]{0.4\textwidth}
	\centering
	\includegraphics*[viewport= 156 188 448 585, width=0.8\textwidth]{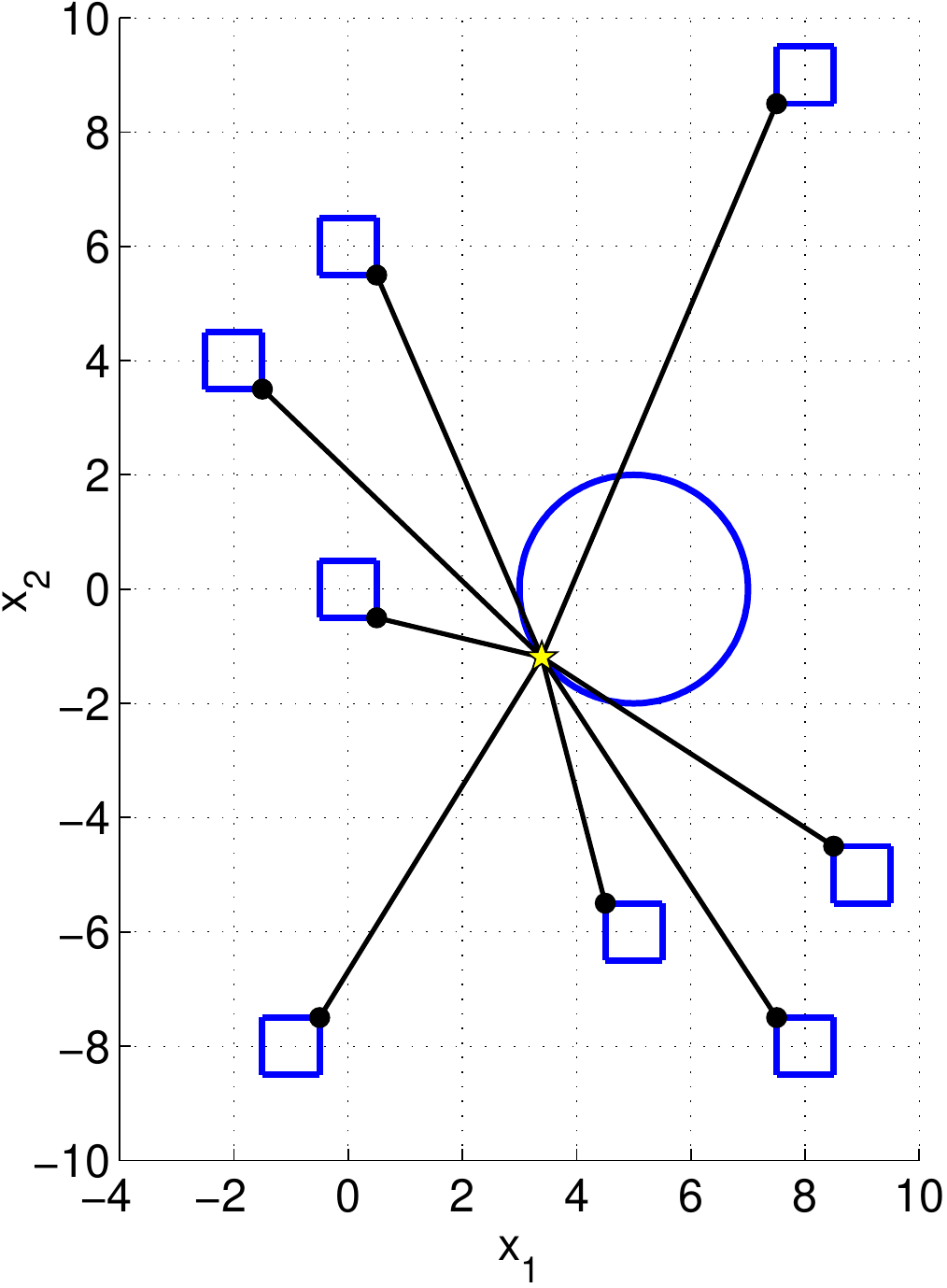}
	\end{minipage}
	\begin{minipage}[t]{0.55\textwidth}
		\centering
		\raisebox{\depth + 3mm}{\small
		\begin{tabular}[htbp]{|c|c|}
				\hline
				\multicolumn {2}{|c|}{Initializations}\\
				\hline
				Algorithm \ref{alg3} & $\sigma_i=0.5$, $\tau=0.24$, \\
				& $\lambda_n=1.8$, $x_0=(5,2)$, \\
			  & $v_{i,0}=0$, $i=1,\ldots,8$ \\ \hline \hline
				Algorithm \ref{alg4} & $\sigma_i=0.1$, $\tau=0.24$, \\
				& $\lambda_n=1.8$, $x_0=(5,2)$, \\
			  & $v_{i,0}=0$, $i=1,\ldots,8$ \\
			\hline
	\end{tabular}
		}
	\end{minipage}
	\caption{\small Example \ref{num-ex1}. Generalized Heron problem with squares and disc constraint set on the left hand side, initializations for the iterative schemes on the right hand side.}
	\label{fig:ex1}	
\end{figure}

\begin{center}
	\small
	\begin{tabular}{| c || c | c | c | c  | } \hline
				& \multicolumn{2}{c|}{Algorithm \ref{alg3}} & \multicolumn{2}{c|}{Algorithm \ref{alg4}} \\\cline{2-3}\cline{4-5}
				$k$   &  $p_{1,k}$ & $V_k$  &  $p_{1,k}$ & $V_k$ \\ \hline \hline
			$0$   &  $(5,-2)$  &  $54.418914$ &  $(5,-2)$ & $54.418914$ \\
			$5$   &  $(3.344027,-1.121496)$  & $53.046330$  & $(3.809999,-1.607451)$ & $53.174978$\\
			$10$  &  $(3.389398,-1.185733)$  & $53.043638$	& $(3.441673,-1.253641)$ & $53.046054$ \\
			$20$  &  $(3.392361,-1.189747)$  & $53.043627$  & $(3.392712,-1.190221)$ & $53.043627$ \\
			$50$  &  $(3.392688,-1.190188)$  & $53.043627$  & $(3.392688,-1.190188)$ & $53.043627$ \\
			$10^6$ & $(3.392688,-1.190188)$  & $53.043627$  & $(3.392688,-1.190188)$ & $53.043627$ \\ \hline
		\end{tabular}
\end{center}

As pointed out in the above table, our methods need less than $50$ iterations to obtain an accuracy up to $6$ decimal places for both function values and iterates. By comparison,
the subgradient-type method in \cite{MorNamSal12b}, when applied to this particular problem with the same starting point, requires more than a million iterations.
\end{example}

\begin{example}[Example 4.3 in \cite{MorNamSal12a}]\label{ex2}
In this example we solve the generalized Heron problem \eqref{ex-p1} in $\R^3$, where the constraint set $\Omega$ is the closed ball centered at $(0,2,0)$ with radius $1$ and $\Omega_i, i=1,...,5,$ are cubes in right position with center at $(0,-4,0)$, $(-4,2,-3)$, $(-3,-4,2)$, $(-5,4,4)$ and $(-1,8,1)$ and side length  $2$, respectively. The plotted  and the initializations for Algorithm \ref{alg3} and Algorithm \ref{alg4} are shown in Figure \ref{fig:ex2}.
\begin{figure}[htb]
	\begin{minipage}[t]{0.4\textwidth}
	\centering
	\includegraphics*[viewport= 71 210 544 581, width=1\textwidth]{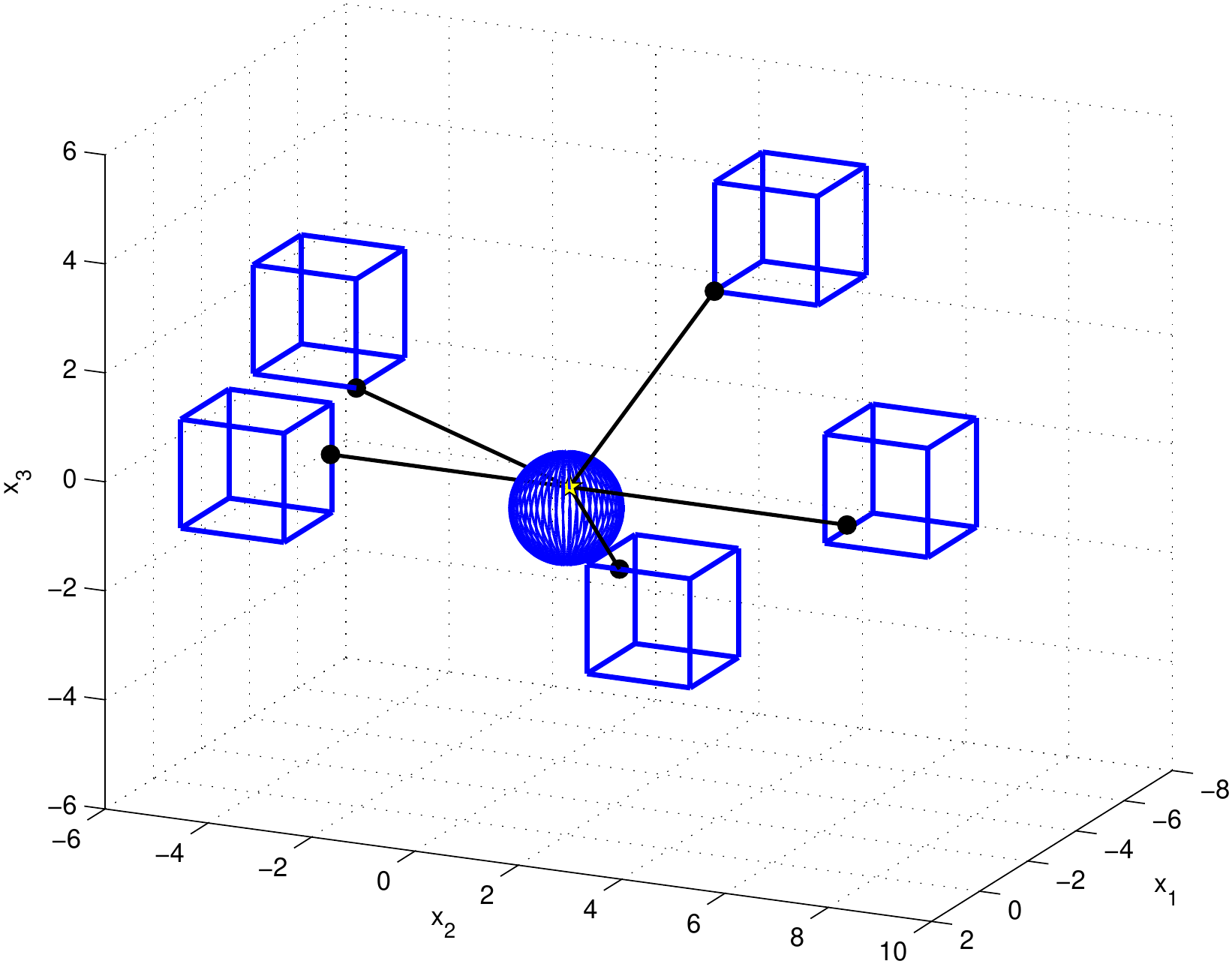}
	\end{minipage}
	\begin{minipage}[t]{0.55\textwidth}
		\centering
		\raisebox{\depth + 3mm}{\small
		\begin{tabular}[htbp]{|c|c|}
				\hline
				\multicolumn {2}{|c|}{Initializations}\\
				\hline
				Algorithm \ref{alg3} & $\sigma_i=0.4$, $\tau=0.99$, \\
				& $\lambda_n=1.8$, $x_0=(0,2,0)$ \\
			  & $v_{i,0}=0$, $i=1,\ldots,8$ \\ \hline \hline
				Algorithm \ref{alg4} & $\sigma_i=0.05$, $\tau=0.59$, \\
				& $\lambda_n=1.8$, $x_0=(0,2,0)$, \\
			  & $v_{i,0}=0$, $i=1,\ldots,8$ \\
			\hline
	\end{tabular}
		}
	\end{minipage}
	\caption{\small Example \ref{ex2}. Generalized Heron problem with cubes and ball constraint set on the left hand side, initializations for the applied methods on the right hand side.}
	\label{fig:ex2}	
\end{figure}
\begin{center}
\small
\begin{tabular}{| c || c | c | c | c | } \hline
		  & \multicolumn{2}{c|}{Algorithm \ref{alg3}} & \multicolumn{2}{c|}{Algorithm \ref{alg4}} \\\cline{2-3}\cline{4-5}
		  $k$   &  $p_{1,k}$ & $V_k$  &  $p_{1,k}$ & $V_k$ \\ \hline \hline
		$0$  &  $(0,2,0)$  & $24.18180$  &  $(0,2,0)$ & $24.18180$ \\
		$5$  &  $(-0.92380,1.62587,0.08140)$  & $22.23482$  & $(-0.93595,1.66118,0.09588)$ & $22.23627$\\
		$10$  & $(-0.92525,1.62890,0.07875)$ &  $22.23480$  & $(-0.92561,1.62957,0.07762)$ & $22.23480$ \\
		$20$  & $(-0.92531,1.62907,0.07883)$   & $22.23480$ & $(-0.92520,1.62880,0.07882)$ & $22.23480$ \\
		$50$  & $(-0.92531,1.62907,0.07883)$  & $22.23480$  & $(-0.92531,1.62907,0.07883)$ & $22.23480$ \\
		$10^6$  & $(-0.92531,1.62907,0.07883)$ & $22.23480$ & $(-0.92531,1.62907,0.07883)$ & $22.23480$ \\ \hline
	\end{tabular}
\end{center}
According to the above table, both algorithms solve this problem in few steps. Especially Algorithm \ref{alg3} achieves accuracies up to $5$ decimal places for both objective function values and iterates in less than $20$ iterations.
\end{example}

\begin{example}[Example 5.4 in \cite{MorNamSal12b}]\label{num-ex3}
Consider again the generalized Heron problem \eqref{ex-p1} for squares in right position in $\R^2$, but this time subject to a real line. More concretely, we take $\Omega=\left\{(x_1,x_2) \in \R^2 : (x_1,x_2) = (1,6) + t(1,0),\ t\in\R \right\}$ and for $\Omega_i, i=1,...,5,$ the squares with center $(-6,-9)$, $(-5,4)$, $(0,-7)$, $(1,0)$ and $(8,8)$ and side length $2$. Figure \ref{fig:ex3} shows the plotted result and the initializations for Algorithm \ref{alg3} and Algorithm \ref{alg4}.
\begin{figure}[htb]
	\begin{minipage}[t]{0.40\textwidth}
	\centering
	\includegraphics*[viewport= 127 187 476 586, width=0.8\textwidth]{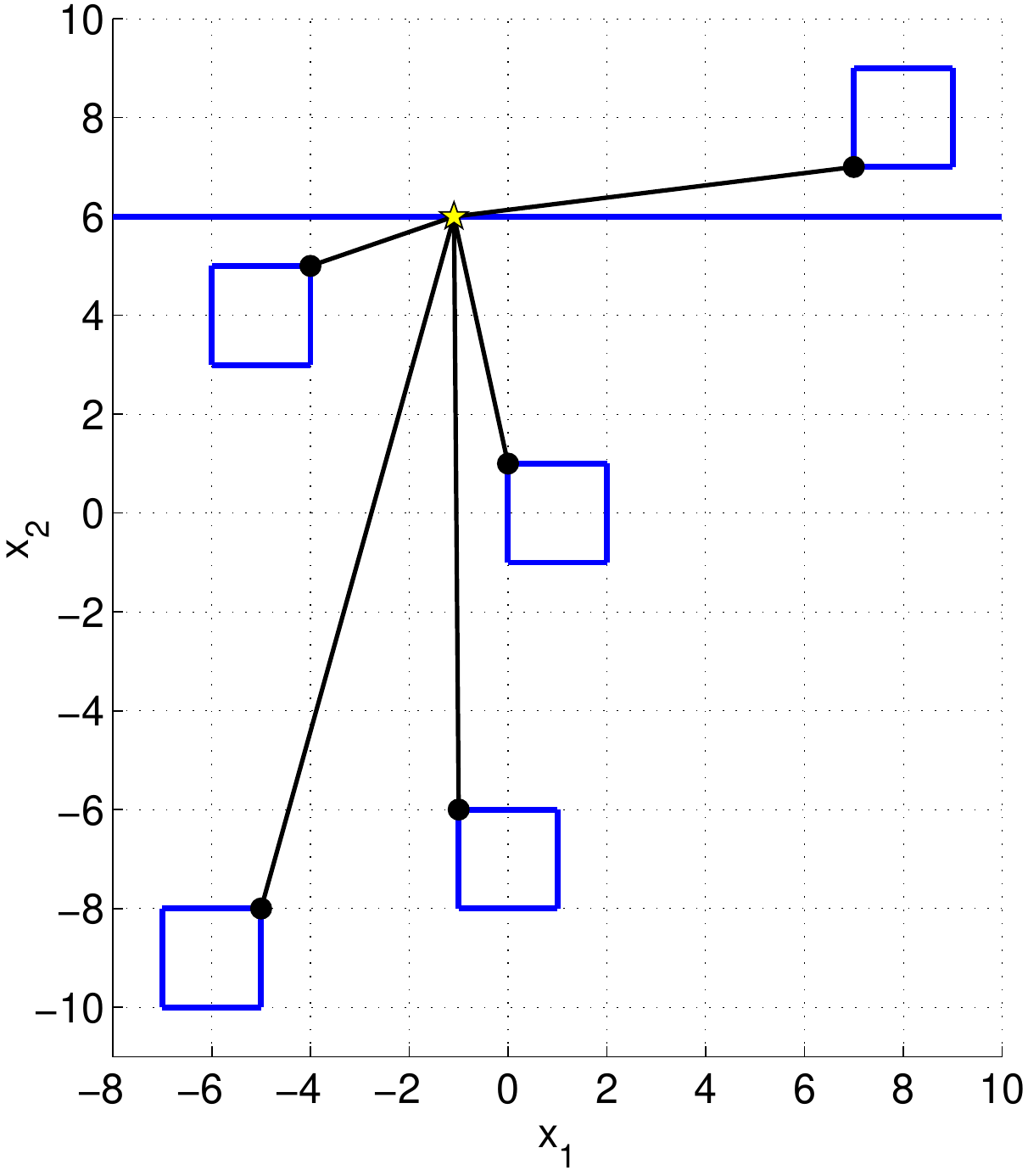}
	\end{minipage}
	\begin{minipage}[t]{0.55\textwidth}
		\centering
		\raisebox{\depth + 3mm}{\small
		\begin{tabular}[htbp]{|c|c|}
				\hline
				\multicolumn {2}{|c|}{Initializations}\\
				\hline
				Algorithm \ref{alg3} & $\sigma_i=0.1$, $\tau=3.99$, \\
				& $\lambda_n=1.7$, $x_0=(-1,6)$, \\
			  & $v_{i,0}=0$, $i=1,\ldots,8$ \\ \hline \hline
				Algorithm \ref{alg4} & $\sigma_i=0.1$, $\tau=0.49$, \\
				& $\lambda_n=1.7$, $x_0=(-1,6)$, \\
			  & $v_{i,0}=0$, $i=1,\ldots,8$ \\
			\hline
	\end{tabular}
		}
	\end{minipage}
	\caption{\small Example \ref{num-ex3}. Generalized Heron problem with squares and line constraint on the left hand side, initializations for the applied methods on the right hand side.}
	\label{fig:ex3}	
\end{figure}
\begin{center}
\small
\begin{tabular}{| c || c | c | c | c  | } \hline
		  & \multicolumn{2}{c|}{Algorithm \ref{alg3}} & \multicolumn{2}{c|}{Algorithm \ref{alg4}} \\\cline{2-3}\cline{4-5}
		  $k$   &  $p_{1,k}$ & $V_k$  &  $p_{1,k}$ & $V_k$ \\ \hline \hline
		$0$   &  $(-1,6)$  & $42.883775$ &  $(-1,6)$ & $42.883775$ \\
		$5$   &  $(-1.215422,6)$  & $42.884811$  & $(-1.136966,6)$ & $42.883775$\\
		$10$  &  $(-1.093321,6)$  & $42.882115$	 & $(-1.107478,6)$ & $42.882444$\\
		$20$  &  $(-1.094633,6)$  & $42.882115$  & $(-1.094886,6)$ & $42.882145$\\
		$50$  &  $(-1.094773,6)$  & $42.882115$  & $(-1.094773,6)$ & $42.882115$\\
		$10^6$ & $(-1.094773,6)$  & $42.882115$  & $(-1.094773,6)$ & $42.882115$ \\ \hline
	\end{tabular}
\end{center}
Algorithm \ref{alg3} and Algorithm \ref{alg4} achieve accuracies for both the values of the objective function and the iterates in less than $50$ iterations. By comparison, as shown in \cite[Example 5.4]{MorNamSal12b}, the subgradient-type methods are not able to provide these accuracies not even after $10$ million iterations.
\end{example}

\subsection{Image deblurring}\label{subsectionImageDeblur}
The second numerical experiment concerns the solving of an ill-conditioned linear inverse problem arising in image deblurring. To this end, we consider images of size $M \times N$ as vectors $x\in\R^n$ for $n=MN$, where each pixel denoted by $x_{i,j}$, $1\leq i \leq M$, $1\leq j \leq N$, ranges in the closed interval from $0$ (pure black) to $1$ (pure white). For a given matrix $A \in \mathbb{R}^{n \times n}$ describing a blur (or averaging) operator and a given vector $b \in \R^n$ representing the blurred and noisy image, our aim is to estimate the unknown original image $\bx\in\R^n$ fulfilling
$$A\bx=b.$$
To this aim we are solving the following regularized convex nondifferentiable problem
 \begin{align}\label{probimageproc}
\inf_{x \in \R^n}{\left\{ \left\| Ax-b \right\|_1 + \alpha_2 \left\| Wx \right\|_1 +\alpha_1 TV(x) + \delta_{\left[0,1\right]^n}(x) \right\}},
\end{align}
where the regularization is done by a combination of two functionals with different properties. Here, $\alpha_1,\,\alpha_2 \in \R_{++}$ are regularization parameters, $TV:\R^n \rightarrow \R$ is the discrete isotropic total variation function and $W:\R^n \rightarrow \R^n$ is the discrete Haar wavelet transform with four levels transforming the image into wavelet coefficients with respect to the orthonormal Haar wavelet basis. Notice that the norm of the operator $W$ is $\|W\|=2^{-8}$ and that none of the functions occurring in the objective function of \eqref{probimageproc} is differentiable.

The picture undergoes a Gaussian blur of size $9 \times 9$ with standard deviation $4$, as done in \cite[Section 4.2]{BotHend12}, yielding  a blurring operator $A$ with $\|A\|^2=1$ and $A^*=A$. For the discrete isotropic total variation functional
\begin{align*}
	TV(x) &= \sum_{i=1}^{M-1}\sum_{j=1}^{N-1}\sqrt{ (x_{i+1,j}-x_{i,j})^2 + (x_{i,j+1}-x_{i,j})^2 } \\
								&\quad + \sum_{i=1}^{M-1} \left| x_{i+1,N}-x_{i,N} \right|  + \sum_{j=1}^{N-1} \left| x_{M,j+1}-x_{M,j} \right|,
\end{align*}
where reflexive boundary conditions are assumed, it holds that $TV(x)=\|Lx\|_{\times}$, where $\Y=\R^n \times \R^n$, the operator $L:\R^n \rightarrow \Y$, $x_{i,j} \mapsto (L_1x_{i,j}, L_2x_{i,j})$,
\begin{align*}
	L_1x_{i,j} = \left\{ \begin{array}{ll} x_{i+1,j}-x_{i,j}, & \text{if }i<M\\ 0, &\text{if }i=M\end{array}\right. \ \mbox{and} \
	L_2x_{i,j} = \left\{ \begin{array}{ll} x_{i,j+1}-x_{i,j}, & \text{if }j<N\\ 0, &\text{if }j=N\end{array}\right.,
\end{align*}
represents a discretization of the gradient in horizontal and vertical direction with $\|L\| \leq 8$ and $\|\cdot\|_{\times}:\Y \rightarrow \R$, $\|(p,q)\|_{\times}=\sum_{i=1}^M \sum_{j=1}^N \sqrt{p_{i,j}^2+q_{i,j}^2}$, is a norm on the Hilbert space $\Y$.

Consequently, the optimization problem \eqref{probimageproc} can be equivalently written as
\begin{align}\label{subsec2-p2}
	\inf_{x \in \R^n}{\left\{f(x) + g_1(Ax) + g_2(Wx) + g_3(Lx) \right\}},
\end{align}
where $f:\R^n \rightarrow \oR$, $f(x)=\delta_{\left[0,1\right]^n}(x)$, $g_1:\R^n \rightarrow \R$, $g_1(y)=\|y-b\|_1$, $g_2:\R^n \rightarrow \R$, $g_2(y)=\alpha_2 \|y\|_1$ and $g_3:\Y \rightarrow \R$, $g_3(y,z) = \alpha_1 \|(y,z)\|_{\times}$. For every $p\in\R^n$ we have $g_1^*(p) = \delta_{\left[-1,1\right]^n}(p) + p^Tb$ and $g_2^*(p)=\delta_{\left[-\alpha_2,\alpha_2\right]^n}(p)$ (see, for instance \cite{Bot10}), while, for every $(p,q)\in\Y$ it holds $g_3^*(p,q) = \delta_S(p,q)$, where (cf. \cite{BotHend12})
$$ S=\left\{ (p,q)\in\Y : \max_{\substack{1\leq i \leq M\\ 1\leq j \leq N }} \sqrt{p_{i,j}^2+ q_{i,j}^2} \leq \alpha_1 \right\}. $$
It is easy to see that for all $x,p,q \in \R^n$ it holds
{\allowdisplaybreaks
\begin{align*}
	\Prox\nolimits_{\tau f}(x) &= \argmin_{z\in\left[0,1\right]^n} \frac{1}{2} \|z - x\|^2 = \proj_{\left[0,1\right]^n}(x)\\
	\Prox\nolimits_{\sigma_1 g_1^*}(p) &= \argmin_{z\in\left[-1,1\right]^n} \left\{\sigma_1z^Tb + \frac{1}{2} \|z-p\|^2\right\} = \proj_{\left[-1,1\right]^n}(p-\sigma_1b) \\
	\Prox\nolimits_{\sigma_2 g_2^*}(p) &= \argmin_{z\in \left[-\alpha_2,\alpha_2\right]^n} \frac{1}{2} \|z-p\|^2 = \proj_{\left[-\alpha_2,\alpha_2\right]^n}(p)\\
	\Prox\nolimits_{\sigma_3 g_3^*}(p,q) &= \argmin_{(z_1,z_2) \in S} \frac{1}{2} \|(z_1,z_2) - (p,q)\|^2 = \proj_S \left(p,q\right),
\end{align*}}
where the projection operator $\proj_S : \Y \rightarrow S$ is defined via
\begin{align*}
		(p_{i,j},q_{i,j}) \mapsto \alpha_1 \frac{(p_{i,j},q_{i,j})}{\max\left\{\alpha_1,\sqrt{p_{i,j}^2 + q_{i,j}^2}\right\}}, \ 1\leq i \leq M,\ 1\leq j \leq N.
\end{align*}
Hence, when choosing $x_0\in\R^n$ and $(v_{1,0},v_{2,0},v_{3,0})\in \R^n \times \R^n \times \Y$ as starting values, and letting $(\lambda_n)_{n\geq 0} \subseteq (0,2)$ and $\tau,\sigma_1,\sigma_2,\sigma_3 \in \R_{++}$ be such that $\tau\left( \sigma_1+ \sigma_2 2^{-16} + 8\sigma_3 \right) < 4$, the iterative scheme in Algorithm \ref{alg3} becomes
	\begin{align*}
	  \left(\forall n\geq 0\right) \begin{array}{l}  \left\lfloor \begin{array}{l}
		p_{1,n} = \proj_{\left[0,1\right]^n}\left(x_n -\frac{\tau}{2}\left(A^*v_{1,n} + W^*v_{2,n} + L^*v_{3,n}\right)\right) \\
		p_{2,1,n} = \proj_{\left[-1,1\right]^n}\left(v_{1,n}+\frac{\sigma_1}{2}A(2p_{1,n}-x_n) - \sigma_1 b\right) \\
		p_{2,2,n} = \proj_{\left[-\alpha_2,\alpha_2\right]^n}\left(v_{2,n} + \frac{\sigma_2}{2}W(2p_{1,n}-x_n) \right) \\
		p_{2,3,n} = \proj_S\left(v_{3,n} + \frac{\sigma_3}{2}L^*(2p_{1,n}-x_n)\right) \\
		z_{1,n} = 2p_{1,n} -x_n - \frac{\tau}{2}\left(A^*(2p_{2,1,n}-v_{1,n}) + W^*(2p_{2,2,n}-v_{2,n})\right.\\
		\hspace{9.5em}\left.+ L^*(2p_{2,3,n}-v_{3,n})\right) \\
		x_{n+1} = x_n + \lambda_n (z_{1,n}-p_{1,n}) \\
		v_{1,n+1} = v_{1,n} + \lambda_n\left(p_{2,1,n} - v_{1,n} + \frac{\sigma_1}{2}A(2z_{1,n}-2p_{1,n}+x_n)\right) \\
		v_{2,n+1} = v_{2,n} + \lambda_n\left(p_{2,2,n} - v_{2,n} + \frac{\sigma_2}{2}W(2z_{1,n}-2p_{1,n}+x_n)\right) \\
		v_{3,n+1} = v_{3,n} + \lambda_n\left(p_{2,3,n} - v_{3,n} + \frac{\sigma_3}{2}L(2z_{1,n}-2p_{1,n}+x_n)\right).
			\end{array} \right.
		\end{array}
	\end{align*}
Similarly, taking also into account Remark \ref{remark-alg4}, Algorithm \ref{alg4} can be also implemented to this problem, this time by choosing $\tau,\sigma_1,\sigma_2,\sigma_3 \in \R_{++}$ such that $\tau\left( \sigma_1+ \sigma_2 2^{-16} + 8\sigma_3 \right) < 1$.

Figure \ref{fig:cameramen-fval-ISNR} shows the performance of  Algorithm \ref{alg3} and Algorithm \ref{alg4} when solving \eqref{subsec2-p2} for $\alpha_1=3$e-$3$, $\alpha_2=2$e-$5$, starting points $x_0=b$ and $(v_{1,0},v_{2,0},v_{3,0})=(0,0,0)$ and parameters
\begin{itemize}
	\setlength{\itemsep}{-2pt}
	\small
	\item DR1 (Algorithm \ref{alg3}): $\sigma_1=1$, $\sigma_2=1$, $\sigma_3=0.05$, $\tau=4\left(\sigma_1+ \sigma_2 2^{-16} + 8\sigma_3\right)^{-1}-0.01$, $\lambda_n=1.5$ for every $n\geq 0$;
	\item DR2 (Algorithm \ref{alg4}): $\sigma_1=1$, $\sigma_2=0.05$, $\sigma_3=0.05$, $\tau=1\left(\sigma_1+ \sigma_2 2^{-16} + 8\sigma_3\right)^{-1}-0.01$, $\lambda_n=1.6$ for every $n\geq 0$,
\end{itemize}
and of the iterative scheme designed in \cite[Theorem 3.1]{ComPes12} for
\begin{itemize}
	\setlength{\itemsep}{-2pt}
	\small
	\item FBF (\cite[Theorem 3.1]{ComPes12}): $\varepsilon=\frac{1}{20(\sqrt{1+2^{-16}+8}+1)}$, $\gamma_n=\frac{1-\varepsilon}{\sqrt{1+2^{-16}+8}}$ for every $n\geq 0$,
\end{itemize}
within the first $200$ iterations, when applied to the $256\times 256$ cameraman test image. Figure \ref{fig:cameraman} shows the original, observed and reconstructed versions of the $256\times 256$ cameraman test image.

\begin{figure}[htb]	
	\centering
	\includegraphics*[width=0.48\textwidth]{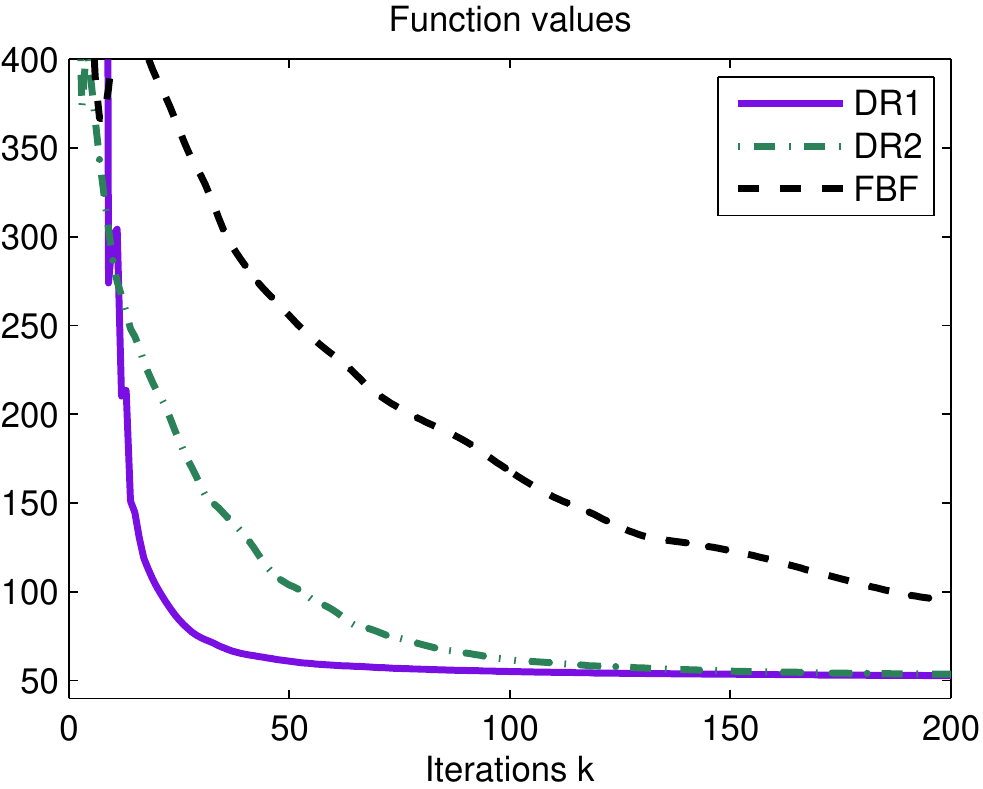}\hspace{2mm}
	\includegraphics*[width=0.48\textwidth]{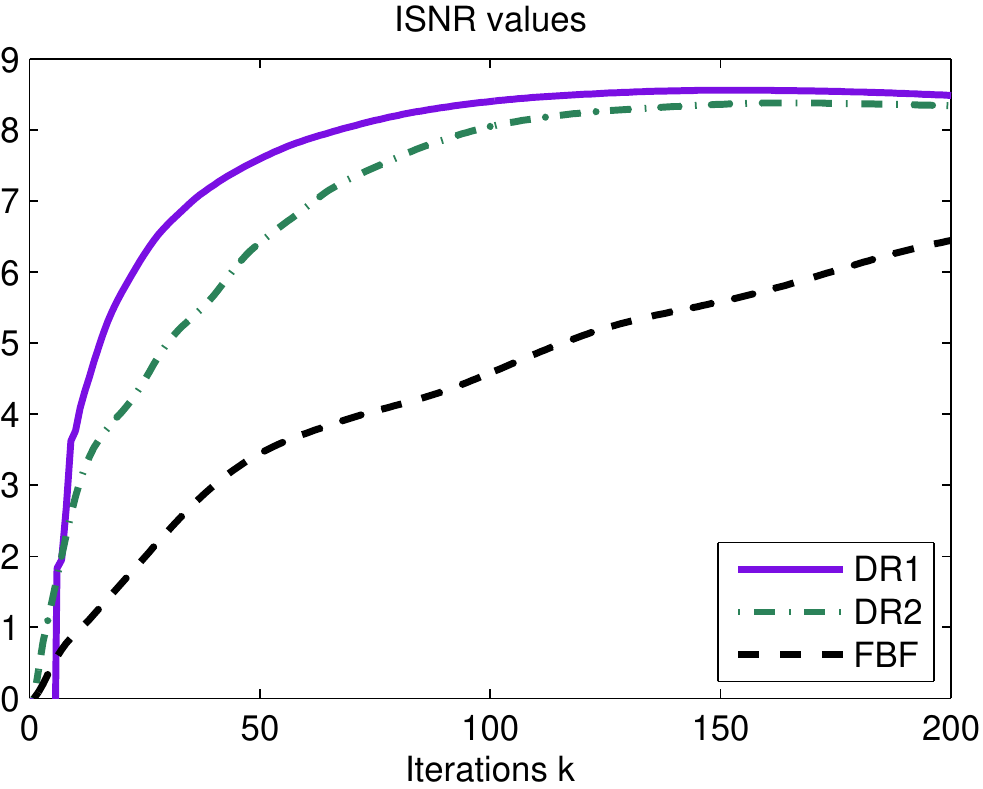}
	\caption{\small The evolution of the values of the objective function and of the ISNR (improvement in signal-to-noise ratio) for Algorithm \ref{alg3} (DR1), Algorithm \ref{alg4} (DR2)  and the forward-backward-forward method (FBF) from \cite[Theorem 3.1]{ComPes12}.}
	\label{fig:cameramen-fval-ISNR}	
\end{figure}
\begin{figure}[htb]	
	\centering
	\captionsetup[subfigure]{position=top}
	\subfloat[Original image]{\includegraphics*[viewport= 144 250 467 574, width=0.32\textwidth]{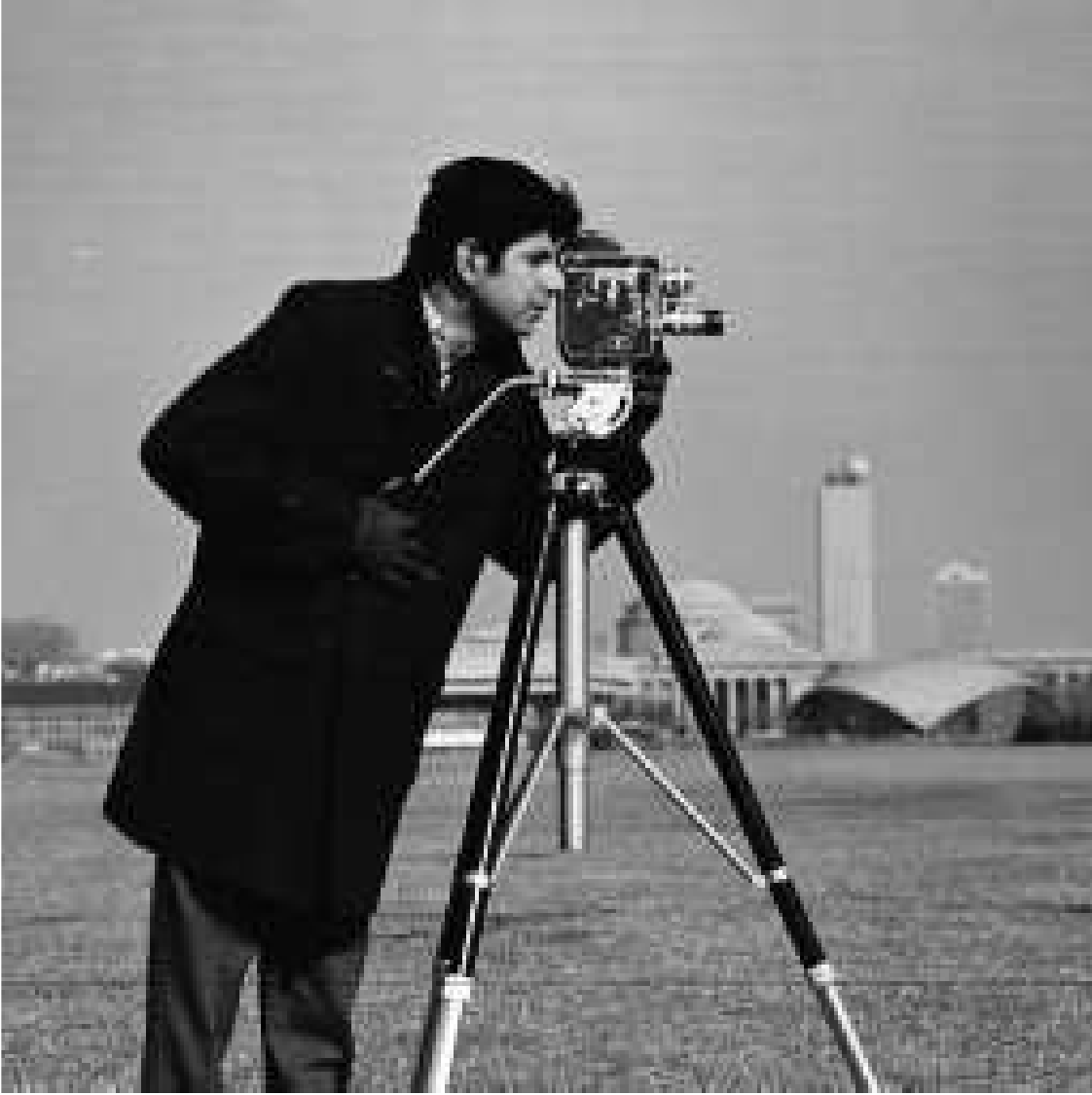}} \hspace{0.2mm}
	\subfloat[Blurred and noisy image]{\includegraphics*[viewport= 144 250 467 574, width=0.32\textwidth]{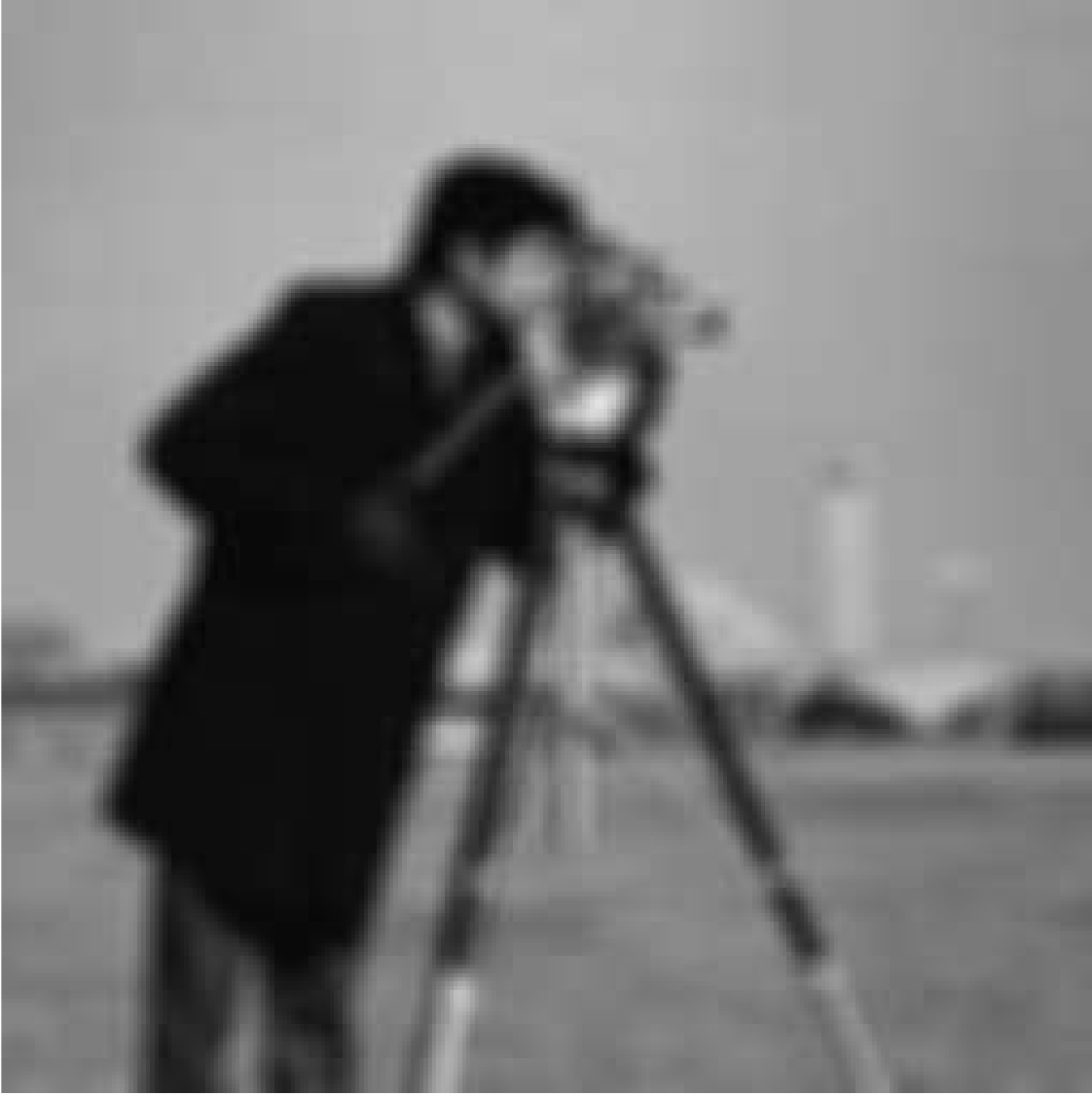}}  \hspace{0.2mm}
	\subfloat[Reconstructed image]{\includegraphics*[viewport= 144 250 467 574, width=0.32\textwidth]{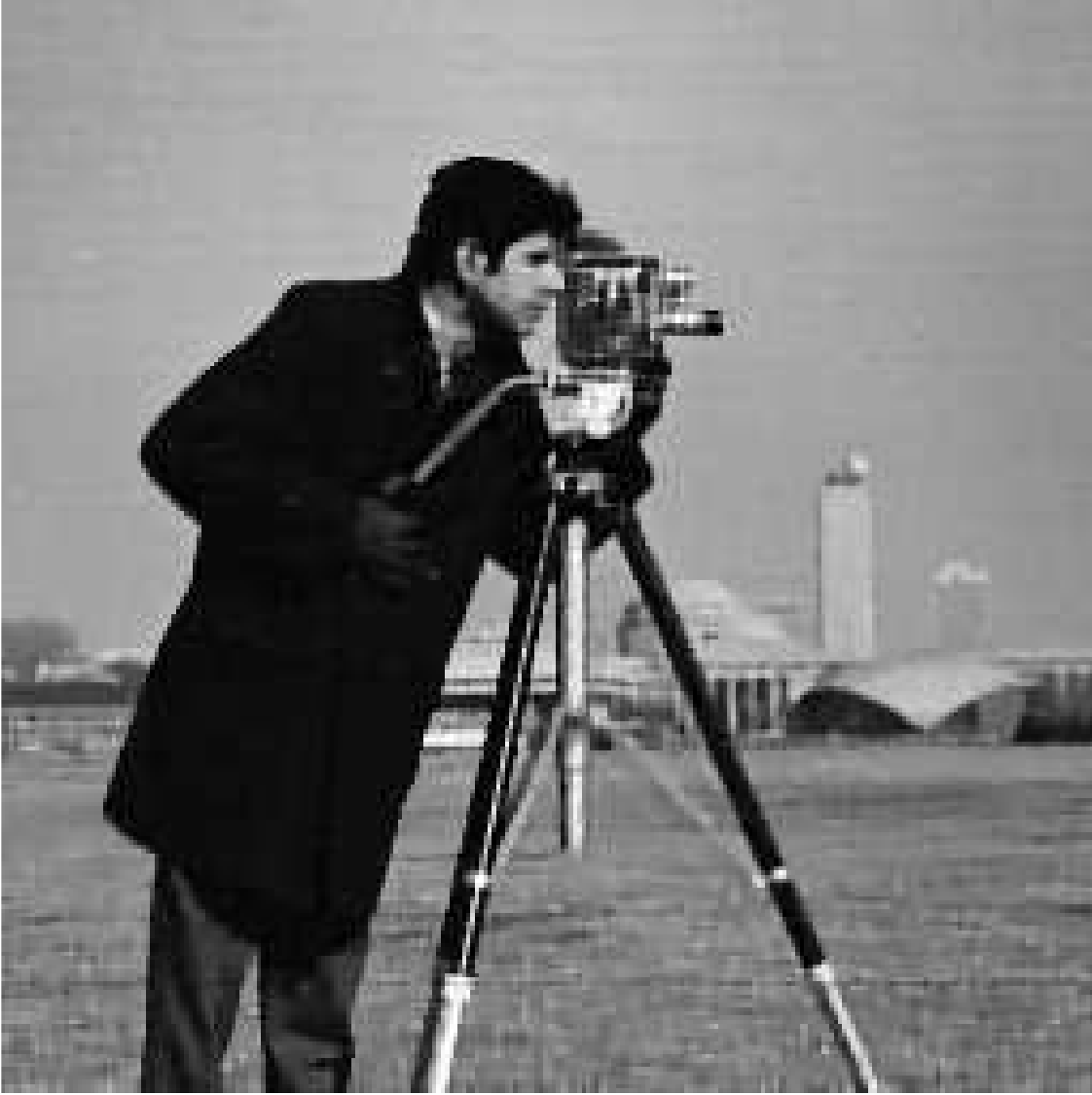}}	
	\caption{\small Figure (a) shows the clean $256\times 256$ cameraman test image, (b) shows the image obtained after multiplying it with a blur operator and adding white Gaussian noise with standard deviation $10^{-3}$ and (c) shows the reconstructed image generated by Algorithm \ref{alg3}.}
	\label{fig:cameraman}	
\end{figure}	

\noindent{\bf Acknowledgements.} The authors are grateful to their colleague E.R. Csetnek for remarks which improved the quality of the paper.

\end{document}